\documentclass[12pt]{amsart}
\usepackage{amsmath}
\usepackage{amsthm}
\usepackage{textcomp}
\usepackage{amsfonts}
\usepackage{txfonts}
\usepackage{mathtools}
\usepackage{amssymb}
\usepackage{newlfont}
\usepackage{galois}
\usepackage{graphicx}
\usepackage[colorlinks,citecolor=blue]{hyperref}

\numberwithin{equation}{section}

\newtheorem{theorem}{Theorem}[section]
\newtheorem{lemma}[theorem]{Lemma}
\newtheorem{remark}[theorem]{Remark}
\newtheorem{corollary}[theorem]{Corollary}

\usepackage{lineno,hyperref}

\begin{document}

\title[Gradient Estimates Under Integral Ricci Curvature Bounds]{Gradient Estimates For $\Delta u + a(x)u\log u + b(x)u = 0$ and its Parabolic Counterpart Under Integral Ricci Curvature Bounds}
\author{Jie Wang}
\address{Institute of Mathematics, Academy of Mathematics and Systems Science, Chinese Academy of Sciences, Beijing 100190, China.}
\email{wangjie9math@163.com}
\author{Youde Wang*}
\address{1. School of Mathematics and Information Sciences, Guangzhou University; 2. Hua Loo-Keng Key Laboratory of Mathematics, Institute of Mathematics, Academy of Mathematics and Systems Science, Chinese Academy of Sciences, Beijing 100190, China; 3. School of Mathematical Sciences, University of Chinese Academy of Sciences, Beijing 100049, China.}
\email{wyd@math.ac.cn}

\thanks{*Corresponding Author}

\begin{abstract}
In this paper, we consider a class of important nonlinear elliptic equations
$$\Delta u + a(x)u\log u + b(x)u = 0$$
on a collapsed complete Riemannian manifold and its parabolic counterpart under integral curvature conditions, where $a(x)$ and $b(x)$ are two $C^1$-smooth real functions. Some new local gradient estimates for positive solutions to these equations are derived by Moser's iteration provided that the integral Ricci curvature is small enough. Especially, some classical results are extended by our estimates and a few interesting corollaries are given. Furthermore, some global gradient estimates are also established under certain geometric conditions. Some estimates obtained in this paper play an important role in a recent paper by Y. Ma and B. Wang \cite{MWW}, which extended and improved the main results due to B. Wang \cite{WB2} to the case of integral Ricci curvature bounds.
\end{abstract}
\keywords{gradient estimate; nonlinear equations; local Sobolev inequality; Moser's iteration}
\maketitle

\vspace{1cm}
\tableofcontents

\newpage
\section{\textbf{Introduction}}
Let $(M,g)$ be a complete Riemannian manifold of dimension $n \geq 2$. Gradient estimate is an essential tool in understanding solutions of nonlinear partial differential equations from geometry. Yau \cite{Yau} and Li-Yau \cite{LY} proved the well-known gradient estimates for Laplace equation $\Delta u=0$ and heat equation $u_t=\Delta u$ defined on $(M, g)$ with Ricci curvature bounded from below, respectively. It is well-known that, by utilizing the gradient estimates of the heat equation, one can obtain Harnack inequalities, the upper bound and the lower bound for the heat kernel, eigenvalue estimate and the lower bound of Green's function on Riemannian manifolds under various geometric conditions on $(M, g)$. The Harnack estimate also plays important role on the Ricci flow (see \cite{H*}).

In this paper, we are concerned with the gradient estimates and related issue for positive solution to following
\begin{equation}\label{equ}
\Delta u+a(x)u \log u+b(x)u=0,
\end{equation}
which is defined on a Riemannian manifold, and its parabolic counterpart
\begin{equation}\label{equ*}
\left( \Delta-\frac{\partial}{\partial t}\right) u+a(x,t)u\log u+b(x,t)u=0.
\end{equation}
In (\ref{equ}), $a(x)$ and $b(x)$ is $C^{1}$-smooth  on $M$, and in (\ref{equ*}), both $ a(x,t) $ and $ b(x,t) $ are $ C^{2} $ with respect to $x\in M$ while are $ C^{1} $ with respect to the time $t$.

The above equation (\ref{equ}) is the close relative of the Euler-Lagrange equations associated with Log Sobolev functional and $\mathcal{W}$-entropy which is a fundamental quantity for Ricci flow introduced by Perelman \cite{P}. Indeed, Perelman has ever made use of the existence of extremals of his $\mathcal{W}$-entropy to prove a no breather theorem stating that shrinking breathers of Ricci flows on compact manifolds are shrinking gradient solitons.

The following definition is one of several equivalent ways in which Perelman's $\mathcal{W}$-entropy on $(M, g)$ can be written. Let $v\in W^{1,2}(M)$ and $\tau>0$ be a parameter. The $\mathcal{W}$-entropy is the quantity
$$W(g,v,\tau) \equiv \int_M\{\tau(4|\nabla v|^2 + Rv^2) - v^2\ln v^2 - \frac{n}{2}(\ln 4\pi\tau)v^2 - nv^2\}dg.$$
Here $n=\dim(M)$ and $R$ is the scalar curvature of $(M, g)$. Let $c > 0$ be a positive constant, it is clear that the $\mathcal{W}$-entropy has the following scaling invariant property
$$\mathcal{W}(cg, c^{-n/4}v, c\tau) = \mathcal{W}(g, v, \tau).$$
Hence we can always take $\tau = 1$ if necessary. If $\tau = 1$ and $\|v\|_{L^2(M)} = 1$, then
$$\mathcal{W}(g,v,\tau) =  \int_M\{(4|\nabla v|^2 + Rv^2) - v^2\ln v^2\}dg - \frac{n}{2}(\ln 4\pi) - n.$$
This means that
$$\mathcal{W}(g,v,\tau) \equiv L(v,g) - \frac{n}{2}(\ln 4\pi) - n.$$
where $$L(v,g)\equiv \int_M\{(4|\nabla v|^2 + Rv^2) - v^2\ln v^2\}dg$$ is the Log Sobolev functional perturbed by the scalar curvature of the manifold $M$. Therefore, the $\mathcal{W}$-entropy and the Log Sobolev functional differ only by a normalizing constant after scaling.
In other words, the $\mathcal{W}$-entropy is just the Log Sobolev functional scaled with certain time dependent parameter. One has known the study of the analytic properties and behaviors of $\mathcal{W}$-entropy's Euler-Lagrange equation is also crucial for understanding deeply the geometry of $(M, g)$, for instance, B. Wang \cite{Wb, WB2} has ever studied the properties and behaviors of Ricci flow via Perelman's local entropy and obtained a series of important results.

For the Log-Sobolev functional, the existence problem of extremal functions in the compact case was solved by O. Rothaus \cite{Ro} 40 years ago.  Later, F. Chung and S.-T. Yau \cite{Ch-Y} also showed that if $M$ is a compact Riemannian manifold of non-positive curvature, then the function $u$ achieving the log-Sobolev constant satisfies a interesting logarithmic Harnack inequality and then used this to prove a lower bound for the log-Sobolev constant of a manifold of non-positive curvature. Now, let us recall the definition of Log-Sobolev functional defined in \cite{Ro}. Let $H$ be a non-negative measurable function on a bounded domain $\Omega\subset M$, for which $\ln H \in L^p$ ($p>\frac{n}{2}$). Let $\rho$ be a positive real number and define $a_\rho(H)$ as the infimum of
$$A_\rho(H)\equiv\int_{\Omega}(\rho|\nabla f|^2 -f^2\ln f^2 + f^2\ln H)d\Omega,$$
for $f \in H^1_0(\Omega)$, subject to the restriction $\int_\Omega f^2=1$. Here, the integral is well defined, since $f \in H^1_0(\Omega)\subset L^{\frac{2n}{n-2}}$. It is easy to see that the minimizer of the constrained variational problem satisfies the following Euler-Lagrange equation
$$\rho\Delta f + f \ln f^2 -f \ln H + a_\rho(H)f = 0.$$
Obviously, the above equation is just a special case of (\ref{equ}).

However, in the noncompact case, the problem is wide open. If $M$ be a complete, connected noncompact manifold with bounded geometry, Zhang \cite{Z*} proved that the Log Sobolev functional $L(v,g)$ has an extremal function decaying exponentially near infinity under a condition near infinity of $(M, g)$. We should mention that, if one drops the connectedness, then it is easy to construct a manifold with infinitely many disconnected components, such that the Log Sobolev functional does not have an extremal. See the example at the beginning of Section 3 in \cite{Z*}.

It is worth to point out that, in addition to being an interesting problem in its own right, the study of Log Sobolev inequality or $\mathcal{W}$-entropy in the noncompact setting is also important to Ricci flow. One reason is that many of the more interesting singularity models are noncompact, even when the Ricci flow under consideration is compact. One such example in the three dimensional case is the round neck $\mathbb{S}^2 \times \mathbb{R}$, which is a typical singularity model. On the other hand, in the case $(M,g)$ is a noncompact gradient shrinking solution, Carrillo and Ni \cite{CN} proved that potential functions are extremals for $\mathcal{W}$-entropy.

In fact, many mathematicians have payed attention to studying the following nonlinear elliptic equation
$$\Delta u + au\log u+bu=0$$
defined on a complete Riemannian manifold $(M, g)$ and its corresponding parabolic equation since the existence problem of extremal functions for the Log-Sobolev functional was solved in the compact case. For example, L. Ma in \cite{Ma} studied the gradient estimates of the positive solutions to the following
$$\Delta u+ au\log u =0 \quad\mbox{in}\,\, M$$
where $(M, g)$ is a complete and $\dim(M)\geq 3$, and obtained some gradient bounds for the case $a<0$ is a constant. Yang \cite{Yang} considered the following
$$\Delta u(x) +au \log u + bu= 0 \quad \mbox{on} \,\, M$$
where $a$ and $b$ are two real numbers, and improves the estimate of \cite{Ma} and extends it to the case $a>0$ and $M$ is of any dimension. Moreover, in \cite{Yang} and \cite{C-C*} they studied
$$u_t=\Delta u(x,t) +au(x,t) \log u(x,t) + bu(x,t) \quad \mbox{on} \,\, M$$
and derived a local gradient estimate for the positive solution of the parabolic equation defined on complete noncompact manifolds with a fixed metric and curvature locally bounded below. For more results on gradient estimates of some related nonlinear equations, we refer to \cite{HM, LY, ML, PWW, PWW1, Wj} and references therein. Note that the above gradient estimates were proved provided the Ricci curvature is bounded from below and the main tool is maximum principle.

On the other hand, there have been some results on gradient estimates under assumptions on integral Ricci curvature bounds(see definitions below) since P. Petersen and G.F. Wei in \cite{PW2} introduced the refined concept on integral Ricci curvature condition. For the gradient estimates, at the very beginning P. Petersen and G.F. Wei in Theorem 3.3 of \cite{PW2} proved a type of gradient estimate for positive harmonic functions on a geodesic ball $B(x,r)$ under some integral Ricci curvature bounds and noncollapsed $M$. In 2018, Dai, Wei and Zhang in \cite{DWZ} obtained the above result for positive harmonic functions without assuming $M$ is non-collapsed by proving a powerful local Sobolev embedding inequality. Q.S. Zhang and M. Zhu \cite{ZZ, ZZ*} also derived a Li-Yau type gradient estimate for the heat kernel of a heat equation under the integral Ricci curvature bounds introduced by Petersen and Wei \cite{PW1, PW2}. Recently, W. Wang in \cite{W} generalizes the gradient estimates for the heat equation $u_t=\Delta u$ in \cite{ZZ, ZZ*} to the case $u_t=\Delta u + au\log u$ where $a$ is a constant.

It is worthy to point out that the local Sobolev inequality and the corresponding De Giorgi-Nash-Moser iteration plays the crucial role on their proofs of the above gradient estimates. For earlier work, we refer to W. Ding and Y. Wang \cite{DiW}, Y. Wang \cite{Wy} and D. Yang \cite{Yd}. For instance, Y. Wang \cite{Wy} has ever shown that a complete manifold is of at least Euclidean volume growth (and hence non-collapsing) if its Sobolev constant is of a positive lower bound.

For more generality, one also studied the gradient estimates of the positive solutions to the nonlinear elliptic equation with variable coefficients (\ref{equ}). Especially, J. Wang in \cite{Wj} has ever derived the gradient estimates for the positive solutions to (\ref{equ}) and (\ref{equ*}) with variable coefficients $a(x,t)$ and $b(x,t)$ if the Ricci curvature is bounded from below.
\medskip

In this paper, we focus on the bound estimates and gradient estimates of the positive solutions to the above (\ref{equ}) and (\ref{equ*}) under integral Ricci curvature bounds. In order to state our results we need to introduce some notions and notations.

In the following, we use $B(x,r)$ or $B_r$ and $\left| B(x, r)\right|$ or $\left| B_r\right| $ to denote the geodesic ball with radius $r$ in $M$ centered at $x$ and its volume, respectively. For any $x\in M$, let $\rho(x)$ be the smallest eigenvalue for the Ricci tensor $Ric:T_xM\rightarrow T_xM$ and $Ric^-$ denote $\max\left\lbrace0,-\rho(x)\right\rbrace $. Then following \cite{DWZ}, for $p,r>0$, we define
$$\kappa(x,p,r)=r^2\left( \fint_{B(x,r)}\left| Ric^-\right|^p \right)^{\frac{1}{p}} ,\quad \quad \kappa(p,r)=\sup\limits_{x\in M}\kappa(x,p,r)$$
and
$$\left| \left| f\right| \right|^*_{p,B(x,r)}= \left( \fint_{B(x,r)}\left| f\right|^p \right)^{\frac{1}{p}}$$ for functions on $M$. Furthermore, when $p\geq1$, it's well-known that the norm $\left| \left| f\right| \right|^*_{p,B(x,r)}$ is non-decreasing in $p$ for fixed $f$ and $B(x,r)$. Since $\left| \left| Ric^-\right| \right|^*_{p,B(x,r)}=0\Longleftrightarrow Ric_M\geq0$, so we always assume $\left| \left| Ric^-\right| \right|^*_{p,B(x,r)}>0$ throughout this paper.

Throughout this paper, let the symbol
$$(f)_+=\max\left\lbrace 0,\, f(x)\right\rbrace , $$
and
$$(g) ^{+}\equiv\sup\limits_{(x,t)\in B(x,1)\times(0,\infty)}\max(g(x,t),\,0).$$

Inspired by the results in \cite{DWZ}, we want to extend the gradient estimates for positive harmonic functions in \cite{DWZ} to the case (\ref{equ}) with non-vanishing coefficients $a(x)$ and $b(x)$ and obtain
\begin{theorem}\label{th1.2}
Let $(M, g)$ be a complete Riemannian manifold, $a,b,\left|\nabla a \right|$ and $\left| \nabla b\right|$ be some bounded functions on $B(x,r)\subset M$ for $0<r\leq1$, i.e., there exist some positive constants $D_1,\,D_2,\, D_3,\, D_4$ such that $$\sup\limits_{B(x,r)}\left| a\right| =D_1,\quad\hspace*{0.3em}\sup\limits_{B(x,r)}\left| b\right| =D_2,\quad\hspace*{0.3em}\sup\limits_{B(x,r)}\left|\nabla a \right|^2=D_3,\quad\hspace*{0.3em}\sup\limits_{B(x,r)}\left|\nabla b \right|^2=D_4.$$ Assume that $u$ is a positive smooth solution to (\ref{equ}) and set $h=\log u$. If $D_5=\left| \left| h\right| \right|^*_{p,B(x,r)}<+\infty$ for $n\geq3$ and $p>\frac{n}{2}$ or $n=2$ and $p>\frac{3}{2}$, then, there exist two constants $C>0$ and $k=k(n,p)>0$ such that if $\kappa(p,1)\leq k$, there holds true
\begin{equation}\label{1.5}
\sup\limits_{B\left( x,\frac{r}{2}\right) }\frac{\left| \nabla u\right|^2}{u^2}\leq C,
\end{equation}
where $C=C\left( n,p,D_1,D_2,D_3,D_4,D_5,\kappa(p,r),C_S,r^{-2}\right)>0$.
Especially, if $a(x)\equiv 0$, the constant $C$ in (\ref{1.5}) is not relevant to $D_1$, $D_3$ and $D_5$; if $a\geq0,\,b\geq0$ on $B_r$, $C$ does not depend on $D_2$ and $D_4$. Here $C_S$ is the local Sobolev constant (see Section \ref{sec2}).
\end{theorem}

In consideration of the monotonicity of the norm $\left| \left| f\right| \right|^*_{p,B(x,r)}$ in $p\geq1$, under stronger assumptions , we can also obtain a different type of gradient estimate and some priori estimates about the positive solutions to equation (\ref{equ}).

\begin{corollary}\label{c1.2'}
Under the same assumptions with $r=1$ as in Theorem \ref{th1.2}, we assume furthermore that $\left\| h\right\|^*_{2p,B(x,1)}=\left\| \log u \right\|^*_{2p,B(x,1)}=D_5'<+\infty,\,D_1\neq0$ and $\sup\limits_{B(x,1)}\left|\Delta a\right| =D_6<+\infty$. Then  there exist two constants $k=k(n, p)>0$ and $C>0$ such that if $\kappa(2p, 1)\leq k$, then
	\begin{equation}\label{1.6*}
		\sup\limits_{B\left( x,\frac{1}{2}\right) }\left( \frac{\left| \nabla u\right|^2}{u^2}+a\log u\right) \leq C,
	\end{equation}
where $C=C\left( n,p,D_1,D_2,D_3,D_4,D_5',D_6,\kappa(2p,1),C_S\right)$. Thus, since $ah\leq C$, there holds true that on $B\left( x,\frac{1}{2}\right) $
	
	(1) if $a\geq A_1$ for some positive constant $A_1$, then
	\begin{equation}\label{1.7*}
		u\leq e^{\frac{C}{A_1}},
	\end{equation}
	
	(2) if $a\leq A_2$ for some negative constant $A_2$, then
	\begin{equation}\label{1.8*}
		u\geq e^{\frac{C}{A_2}}.
	\end{equation}
\end{corollary}

Moreover, under some specific conditions, we have the following result which does not depend on $D_5$.
\begin{corollary}\label{c1.1}
Under the same assumptions as in Theorem \ref{th1.2}, we assume furthermore that $u\leq D$ for some positive constant $D$ is a bounded positive smooth solution to (\ref{equ}), $a\leq0$ is constant, $b$ and $\left| \nabla b\right|$ are bounded on $B(x,r)\subset M$ for $0<r\leq1$. Then, for $n\geq3$ and $p>\frac{n}{2}$ or $n=2$ and $p>\frac{3}{2}$, there exist two constants $C>0$ and $k=k(n,p)>0$ such that, if $\kappa(p,1)\leq k$ there holds true
	\begin{equation}\label{1.7'}
	\sup\limits_{B\left( x,\frac{r}{2}\right) } \frac{\left| \nabla u\right|^2}{u^2}\leq C,
	\end{equation}
where $C=C\left( n,p,D_1,D_2,D_4,D,\kappa(p,r),C_S,r^{-2}\right)>0$. Especially, if $a\log D+b\geq0$ on $B_r$, $C$ does not depend on $D_2$, $D_4$ and $D$.
\end{corollary}

On the other hand, it's well-known that when $a(x)>0$ is constant, equation (\ref{equ}) is related closely to Perelman's $\mathcal{W}$-entropy and Log-Sobolev functional $\mathcal{L}(u, M, g)$ for $u\in W^{1,2}(M)$, in this situation, $b(x)$ is linked with scalar curvature and the best Log-Sobolev constant. For details about correlative backgrounds, see \cite{RV, Z*}.

The deep connections between the equation (\ref{equ}) with  constant $a(x)\equiv a>0$ and many significant areas in geometry and analysis reveal that this case does matter. Actually, in view of Corollary \ref{c1.1}, we can see that the case of $a\leq0$ is much simpler than $a>0$ in some sense.

For the case of $a>0$, for a solution $u\in W^{1,2}_{loc}(M)$ to (\ref{equ}), Q.S. Zhang in \cite{Z*} obtained the bounds and gradient estimates provided assumptions of bounded geometry and non-collapsing conditions. In the present paper, we can also extend these results to the case of integral Ricci curvature bounds with non-collapsed $(M, g)$.
\begin{theorem}\label{th1.2'}
Let $(M, g)$ be a Riemannian manifold, $u\in W^{1,2}(B_1)$ be a positive solution to (\ref{equ}) and $a>0$ be constant. Also, assume that $\inf\limits_{x\in M}\left| B\left( x, \frac{1}{2}\right) \right|\geq V$ for some constant $V>0$. Then, for $n\geq3$ and $p>\frac{n}{2}$ or $n=2$ and $p>\frac{3}{2}$ there exists a constant $k=k(n,p)>0$ such that, when $\kappa(p,1)\leq k$, we have
	
	(1) if $\left| \left|(b)_+ \right| \right|^*_{p, B_1}<+\infty$, there holds
	\begin{equation}\label{1.8**}
	\sup\limits_{B\left( x,\frac{1}{2}\right) } u\leq D,
	\end{equation}
	where $D=D\left( n,p,a,V,C_S,\left| \left|(b)_+ \right| \right|^*_{p, B_1},\left| \left| u\right| \right|_{2,B_1} \right)$ is some positive constant;
	
	(2) if $\left| \left|\left| \nabla b\right|^2  \right| \right|^*_{p, B_{1}}<+\infty$ and $\left| \left|(b)_+ \right| \right|^*_{p, B_1}<+\infty$, there holds
	\begin{equation}\label{1.9**}
	\sup\limits_{B\left( x,\frac{1}{2}\right) } \left|\nabla u\right|^2 \leq C,
	\end{equation}
	 where $C=C\left( n,p,a,V,\kappa(p,1),C_S,\left| \left|(b)_+\right| \right|^*_{p, B_1},\left| \left|\left| \nabla b\right|^2  \right| \right|^*_{p, B_{1}},\left| \left| u\right| \right|_{W^{1,2}(B_1)} \right)$ is some positive constant.
\end{theorem}

However, when $a>0$ is constant and $u$ is only a bounded solution to equation (\ref{equ}), we can not expect to obtain a estimate like (\ref{1.5}) since $D_5$ maybe infinite. But in Proposition 3.1 of \cite{TW}, under the conditions that $\left| B_1\right| $ and $Ric_{M}$ are bounded from below, the authors obtain a type of local H\"{o}lder continuity of $\left| \nabla u\right|$ by using maximum principle to get a estimate of $\left| \nabla u\right|/u$.

Later in \cite{Wb, WB2}, B. Wang studied the properties of Ricci flow via Perelman's local entropy and obtained a series of important results. Especially, in his work \cite{WB2} the gradient estimate in \cite{TW} is indispensable to obtain the crucial improved pseudo-locality theorem. In order to generalize the main results in \cite{WB2} to the case of integral Ricci curvature bounds, we need to obtain an analogous gradient estimate and this is actually our original motivation to study the gradient estimate of the positive solutions to (\ref{equ}). In fact, by similar methods with that to prove Theorem \ref{th1.2}, we can obtain the following estimate under integral Ricci curvature bounds. Utilizing this estimate, Y. Ma and B. Wang \cite{MWW} extended successfully the main results in \cite{WB2} to the case of integral Ricci curvature bounds. For more information about the relations between equation (\ref{equ}) and local entropy, we refer to \cite{MWW, Wb, WB2}.

\begin{theorem}\label{th1.3'}
Assume $a\geq0$ is constant and $0<u\leq D$ is a positive solution to (\ref{equ}) where $D$ is some positive number. If $\left| \left|b\right| \right|^*_{p, B_\lambda}<+\infty$ and $\left| \left|\left| \nabla b\right|^2\right| \right|^*_{p, B_\lambda}<+\infty$, then, for $n\geq3$ and $p>\frac{n}{2}$ or $n=2$ and $p>\frac{3}{2}$, there exists a constant $k=k(n,p)>0$ such that if $\kappa(p,1)\leq k$, the following holds true:

For any $0<\lambda\leq1$ and $q>1$, we have
\begin{equation}\label{1.10''}
\sup\limits_{B\left( x,\frac{\lambda}{2}\right) }\frac{\left|\nabla u\right|}{u^{1-\frac{1}{2q}}} \leq \frac{C}{\lambda}
\end{equation}
where the constant $C=C\left( n,p,q,D,a,C(n),\kappa(p,1),\left| \left|b\right| \right|^*_{p, B_1},\left| \left|\left| \nabla b\right|^2 \right|\right|^*_{p, B_1}\right) $ and $C(n)$ is determined by $C_S$(see Theorem \ref{th1.1}).
\end{theorem}

\begin{remark}\label{r1.1'}
As a direct consequence of (\ref{1.10''}), we obtain (\ref{1.9**}) at once without assuming $\left| \nabla u\right|\in L^2(B_1)$ by setting $\lambda=1$. But here the condition $\left| \left|(b)_+ \right| \right|^*_{p, B_1}<+\infty$ needs to be replaced by $\left| \left|b \right| \right|^*_{p, B_1}<+\infty$.
\end{remark}

Up till now, all estimates obtained in the above are local. A natural problem is whether or not we can obtain some global results which are similar to (\ref{1.10''}). Indeed, we can show the following theorem which can be regarded as the extension of some classical gradient estimates.

\begin{theorem}\label{th1.4'}
Let $(M,g)$ be a complete non-compact Riemannian manifold of dimension $n\geq3$ without boundary, and its Ricci curvature $Ric_M\geq0$. Let $0<u\leq D$ be a solution to (\ref{equ}) with constant coefficients $a(x)\equiv a\geq0$ and $b(x)\equiv b$, where $D$ is some positive number. Then, for any $q>1$ and $R>0$, there exists a uniform constant $C=C(n,q,a,b,D)$ which does not depend on the radius $R$ such that
\begin{equation}\label{1.11'}
\sup\limits_{B\left( x,\frac{R}{2}\right) }\frac{\left|\nabla u\right|}{u^{1-\frac{1}{2q}}}\leq\frac{C}{R},
\end{equation}
if $\,a\log D+2aq+b\leq0$. Consequently, in this situation, $u$ must be constant on $M$. In other words, if $a>0$, then $u$ must be constant if the upper bound of $u$ is small enough.
\end{theorem}

\begin{remark}\label{r1.8}
	From the following proof of Theorem \ref{th1.4'}, obviously, one can see that if $a=0$ and $b\leq 0$ simultaneously, then we need not to assume that $u$ has a upper bound. In fact, by a scaling argument, (\ref{1.11'}) also holds true as long as $\kappa(p,R)<\varepsilon(n,p)$. But as $R\longrightarrow\infty$, in view of Lemma \ref{l2.2}, $\kappa(p,R)<\varepsilon(n,p)$ implies $Ric_M\geq0$.
\end{remark}

As a direct consequence of Theorem \ref{th1.4'} and Remark \ref{r1.8}, we obtain the following classical result again. Specifically, letting $a=0$, $b=0$ and $R\rightarrow+\infty$, we have the following
\begin{corollary}
Let $(M,g)$ satisfy the same assumptions as in Theorem \ref{th1.4'}. Then, any positive harmonic functions on such $M$ must be constant.
\end{corollary}

As for the equation (\ref{equ*}), under inspirations from \cite{W} and \cite{ZZ}, we can also prove the following result:
\begin{theorem}\label{th1.3}
Let $u$ be a positive smooth solution to (\ref{equ*}) and satisfy $u\leq D$ for some positive constant $D$. Let $f=\log\frac{u}{D}$ and assume that, on $B(x,1)\times(0,\infty)$, $a, b, \left| \nabla a\right|, \left| \nabla b\right|$ and $\left|a_t \right|$ are bounded and $\Delta b$ is bounded from below. For manifold $(M,g)$, we assume dimension $n\geq2$ and $p>\frac{n}{2}$. Furthermore, let $N<0$ be a constant depends on $n$, $D$ and the bounds of $a(x, t)$, and $A$ be a constant which satisfies $ A>(a)^+$ and $A\geq(-a)^++3$. Then, on $B\left( x,\frac{1}{2}\right) \times(0,\infty)$, there exists a constant $k=k(n,p)>0$ such that if $\kappa(p,1)\leq k$, it holds that
\begin{equation}\label{1.7}
\begin{split}
&\underline{J}\left| \nabla f\right|^2+(A+a)f+2(N+b)-2f_t\\
&\leq \frac{4n}{(2-\delta)\underline{J}}\left\lbrace \frac{n\left( \left|\nabla a \right|^2 \right) ^+}{(2-\delta)\underline{J}(A-(a)^+)^2}+\frac{(\Delta a+a_t)^+}{(A-(a)^+)}+\frac{8nC}{(2-\delta)\underline{J}}+(a)^++8C+\frac{1}{t}\right\rbrace
\end{split}
\end{equation}
where
$C$ is the same as in Lemma \ref{l1.1}, $\delta$, $C_1$ and $C_2$ are some constants such that $$0<\delta\leq\frac{2}{1+4n}, \quad C_1=\frac{5}{\delta},\quad C_2=C_2(n,p)$$
and
$$\underline{J}=\underline{J}(t)=2^{\frac{-1}{C_1-1}}\exp\left\lbrace -2C_2k\left( 1+\left[ 2C_2(C_1-1)k\right]^{\frac{n}{2p-n}} \right)t \right\rbrace.$$
\end{theorem}
\begin{remark}\label{r1.1}
From the following proof of Theorem \ref{th1.3}, it's easy to observe that if $a(x,t)$ is a constant, e.g., $a(x,t)\equiv A$, then we do not need to assume $u$ has a upper bound and hence recover the estimate in \cite{W}.
\end{remark}
\begin{remark}\label{r1.2}
	As pointed out in Section 2.3 of \cite{PW2}, for $r_2>r_1$, $\kappa(p, r_2)$ and $\kappa(p, r_1)$ can be controlled by each other via multiplying by constants which depend on $n,p,r_2,r_1$. This is why we can always assume $\kappa(p, 1)\leq k$ rather than $\kappa(p, r)\leq k$ for $r<1$.  Especially, $ \kappa(p,r_1)\leq 2^{\frac{1}{p}}\left( \frac{r_1}{r_2}\right)^{2-\frac{n}{p}}\kappa(p,r_2)\leq2^{\frac{1}{p}}\kappa(p,r_2)$.
\end{remark}

\section{\textbf{Preliminaries and Notations}}\label{sec2}
In the following arguments presented in this paper, we need to use the following theorem which claims the local Sobolev embedding is valid under some geometric conditions on integral Ricci curvature bounds.
\begin{theorem}[{\cite{DWZ}}, Corollary 4.6]\label{th1.1}
For $p>\frac{n}{2}$ and $0<r\leq1$, there exist a constant $k=k(n,p)>0$ and some $C_S=C_S(n,r)$ such that if $\kappa(p,1)\leq k$, there hold that
\begin{equation}\label{1.3}
\left| \left| f\right| \right|^*_{\frac{n}{n-1},B(x,r)}\leq C_S\left| \left| \nabla f\right| \right|^*_{1,B(x,r)}
\end{equation}
and
\begin{equation}\label{1.4}
\left| \left| f\right| \right|^*_{\frac{2n}{n-2},B(x,r)}\leq C_S\left| \left| \nabla f\right| \right|^*_{2,B(x,r)}
\end{equation}
for any $f\in C^{\infty}_0(B(x,r))$. Here we need to point out that $C_S$ in (\ref{1.3}) and (\ref{1.4}) are different from each other, but for simplicity, we use the same symbol. Also, $C_S$ has the form $C_S=C(n)r$, so $C_S$ is in the scale proportional to $r$.
\end{theorem}

In particular, we also need to use the following lemma established in \cite{DWZ} as they derived the gradient estimates on positive harmonic functions in \cite{DWZ, PW1}.
\begin{lemma}[{\cite{DWZ}}, Lemma 5.4]\label{l1.1}
For $p>\frac{n}{2}$ and $0<r\leq1$, there exist a constant $k=k(n,p)>0$ and some constant $C=C(n,p)$ such that if $\kappa(p,1)\leq k$, then there exists a cut-off function $\phi\in C^\infty_0\left(  B(x,r)\right)$ satisfies
\begin{equation}\label{1.6}
0\leq\phi\leq1,\quad \phi\equiv1 \hspace*{0.3em}\mbox{on} \hspace*{0.3em}B\left( x,\frac{r}{2}\right)\quad \mbox{and}\quad \left| \nabla\phi\right|^2+\left| \Delta\phi\right| \leq\frac{C}{r^2}.
\end{equation}
\end{lemma}

Given $x\in M$, let $r(y)=d(y,x)$ be the distance function from $x$ and $\psi(y)=\left( \Delta r-\frac{n-1}{r}\right) _+$. The classical Laplacian comparison states that, if the Ricci curvature $Ric_M$ of $M$ satisfies $Ric_M\geq0$, then $\Delta r\leq \frac{n-1}{r}$, i.e. $\psi\equiv0$. In \cite{PW1}, this reslut is generalized to the case of integral Ricci bounds:
\begin{lemma}[{\cite{PW1}}, Lemma 2.2]\label{l2.1}
For $p>\frac{n}{2}$ and $r>0$, there holds
\begin{equation*}
\left| \left| \psi\right| \right|_{2p,B(x,r)}\leq\left(\frac{(n-1)(2p-1)}{2p-n} \left| \left|Ric^- \right| \right|_{p,B(x,r)} \right)^{\frac{1}{2}},
\end{equation*}
equivalently,
\begin{equation}\label{2.1}
\left| \left| \psi\right| \right|^*_{2p,B(x,r)}\leq\left(\frac{(n-1)(2p-1)}{2p-n} \left| \left|Ric^- \right| \right|^*_{p,B(x,r)} \right)^{\frac{1}{2}}.
\end{equation}	
\end{lemma}

We also need to use the so-called volume doubling property as following
\begin{lemma}[{\cite{PW2}}, Theorem 2.1]\label{l2.2}
For any $p > n/2$ there exists a constant $k = k(n, p) $
such that if $\kappa(p, r)\leq k$, then for any $x\in M$ and $0<r_1 < r_2 \leq r$, we have
\begin{equation}\label{2.2}
\frac{\left| B(x,r_2)\right| }{r_2^n}\leq\frac{2\left| B(x,r_1)\right| }{r_1^n}.
\end{equation}
\end{lemma}

As we know, the norm $\left| \left| f\right| \right|^*_{p,B(x,r)}$ is non-decreasing in $p\geq1$ for fixed $f$ and $B(x,r)$, but by the proof of Theorem 2.1 of \cite{LS}, a type of inverse inequality about $p$ also holds true under a maximum condition and volume doubling property. The following is a refined result from Theorem 2.1 in \cite{LS}, actually, it has been used in the proof of Theorem 5.3 of \cite{DWZ}.
\begin{lemma}\label{l2.3}
Let the manifold $M$ satisfy volume doubling property. If for any $\frac{1}{2}\leq\theta\leq\frac{4}{5}-\delta$ with $0<\delta\leq\frac{1}{2}$, there exist some constants $K_1$, $K_2$, $t>s>0$ and $R>0$ such that the positive function $v$ satisfies
\begin{equation}\label{2.3'}
\sup\limits_{B(x,\theta R)}v\leq \left( K_1\delta^{-K_2}\right) ^{\frac{1}{t}}\left| \left| v\right| \right|^*_{t,B\left( x,\left( \theta+\delta\right)R \right) },
\end{equation}
then there exists a constant $C=C(K_1,K_2,n,s,t,\tau)$ such that
\begin{equation}\label{2.4'}
\sup\limits_{B\left( x,\left( \frac{4}{5}-\tau\right)  R\right) }v\leq C\left| \left| v\right| \right|^*_{s,B\left( x,R \right) }
\end{equation}
for any constant $0<\tau\leq\frac{3}{10}$.
\end{lemma}
\begin{proof}
Let $B_r$ and $\left|B_r \right|$ denote $B(x, r)$ and $\left|B(x, r) \right|$ respectively. By (\ref{2.3'}), we have
\begin{equation*}
\sup\limits_{B_{\theta R}}v^t\leq K_1\delta^{-K_2}\fint_{B_{\left( \theta+\delta\right)R}} v^t,
\end{equation*}
since $\theta+\delta\geq\frac{1}{2}$, then there holds
\begin{equation}\label{2.8'}
\sup\limits_{B_{\theta R}}v^t\leq K_1\delta^{-K_2}\left| B_{\frac{R}{2}}\right|^{-1} \int_{B_{\left( \theta+\delta\right)R}} v^t.
\end{equation}
On the other hand, we have
\begin{equation}\label{2.9'}
\int_{B_{\left( \theta+\delta\right)R}} v^t\leq\sup\limits_{B_{\left( \theta+\delta\right)R}}v^{t-s}\int_{B_{\left( \theta+\delta\right)R}} v^s\leq\left( \sup\limits_{B_{\left( \theta+\delta\right)R}}v^t\right)^{\frac{t-s}{t}} \int_{B_\frac{4R}{5}} v^s.
\end{equation}
Now, we set $$M(\theta)=\sup\limits_{B_{\theta R}}v^t\quad\quad \mbox{and}\quad\quad N=\left| B_{\frac{R}{2}}\right|^{-1} \int_{B_\frac{4R}{5}} v^s.$$
From (\ref{2.8'}) and (\ref{2.9'}) we get
\begin{equation}\label{2.10'}
M(\theta)\leq NK_1\delta^{-K_2}\left( M(\theta+\delta)\right)^{\frac{t-s}{t}}.
\end{equation}
Choosing $\theta_0=\frac{4}{5}-\tau\geq\frac{1}{2}$, $\delta_{i-1}=2^{-i}\tau$ and $\theta_i=\theta_{i-1}+2^{-i}\tau$ for $i=1,2,3,\dots\ldots$, and denoting $\frac{t-s}{t}$ by $\lambda$, then we have
\begin{equation*}
M(\theta_{i-1})\leq 2^{iK_2}\tau^{-K_2}NK_1\left( M(\theta_i)\right)^\lambda.
\end{equation*}
Iterating step by step, for positive integer $j$ it holds that
\begin{equation}\label{2.11'}
M(\theta_0)\leq2^{K_2\sum^{j}_{i=1}i\lambda^{i-1}}\left(\tau^{-K_2} NK_1\right)^{\sum^{j}_{i=1}\lambda^{i-1}}M(\theta_j)^{\lambda^j}.
\end{equation}
Letting $j\rightarrow+\infty$, then $\theta_j\rightarrow \frac{4}{5}$, $\lambda^j\rightarrow 0$ and $\sum^{j}_{i=1}\lambda^{i-1}\rightarrow \frac{t}{s}$. Then, from (\ref{2.11'}) it turns out that
\begin{equation*}
M(\theta_0)\leq C(K_2, \tau,t,s)K_1^{\frac{t}{s}}\left( \left| B_{\frac{R}{2}}\right|^{-1} \int_{B_\frac{4R}{5}} v^s\right)^\frac{t}{s},
\end{equation*}
and it leads to
\begin{equation}\label{2.12'}
\sup\limits_{B_{\left( \frac{4}{5}-\tau\right) R}}v\leq C(K_2, \tau,t,s)^{\frac{1}{t}}K_1^{\frac{1}{s}}\left( \left| B_{\frac{R}{2}}\right|^{-1} \int_{B_\frac{4R}{5}} v^s\right)^{\frac{1}{s}}.
\end{equation}
By Lemma \ref{l2.2}, there holds
\begin{equation*}
\left| B_{\frac{R}{2}}\right|^{-1}\leq2\left( \frac{8}{5}\right)^n \left| B_\frac{4R}{5}\right|^{-1},
\end{equation*}
then we obtain (\ref{2.4'}) at once.
\end{proof}

\section{\textbf{Proofs of Theorem \ref{th1.2}, Corollary \ref{c1.2'} and  \ref{c1.1}}}
Now we are ready to give the whole proof of Theorem \ref{th1.2} and the basic ideas of iteration follow from  \cite{DWZ, DiW, Wy}.

\begin{proof}[\textbf{Proof of Theorem \ref{th1.2}}]
By scaling we may assume $r=1$. For simplicity, we denote $B(x,r)$
by $B_r$, $\left| Ric^-\right|$ by $R_c$, and if there has no special emphasis on integral domain, we always calculate on $B_1$ and set
$$h=\log u.$$
Let $N=\sup\limits_{B_1}\left| b\right|+1 $, and $v=\left| \nabla h\right|^2+N$, then (\ref{equ}) gives
\begin{equation}\label{2.5}
\Delta h+\left| \nabla h\right|^2+ah+b=0,
\end{equation}
and
\begin{equation}\label{2.6}
\Delta h=N-v-ah-b.
\end{equation}
By Bochner formula, we have
\begin{align}\label{2.7}
\frac{1}{2}\Delta v=\frac{1}{2}\Delta \left| \nabla h\right|^2&=\left|D^2h \right|^2+\left\langle \nabla h, \nabla\Delta h\right\rangle +Ric\left(\nabla h, \nabla h\right)\nonumber\\
&\geq\frac{(\Delta h)^2}{n}+\left\langle \nabla h, \nabla\Delta h\right\rangle-R_c\left| \nabla h\right|^2.
\end{align}
Substituting (\ref{2.6}) into (\ref{2.7}) and noting $\left| \nabla h\right|^2=v-N$, we have
\begin{equation}\label{2.8}
\begin{split}
\frac{1}{2}\Delta v\geq&\frac{(N-v-ah-b)^2}{n}-\left\langle \nabla h, \nabla v\right\rangle-a(v-N)\\
&-h\left\langle \nabla h, \nabla a\right\rangle-\left\langle \nabla h, \nabla b\right\rangle-R_c(v-N).
\end{split}
\end{equation}
By virtue of the following inequalities:
$$(ah+b-N)^2\geq0,$$
$$h\left\langle \nabla h, \nabla a\right\rangle\leq\frac{\left|h \right| }{2}\left( v-N+\left| \nabla a\right|^2 \right),$$
$$\left\langle \nabla h, \nabla b\right\rangle\leq\frac{1 }{2}\left( v-N+\left| \nabla b\right|^2 \right),$$
and noting $R_cN\geq0$, $\left|h \right|N\geq0$, $a\geq-\left|a \right|$ and $N \geq 0$, then, from (\ref{2.8}) we derive
\begin{equation}\label{2.9}
\begin{split}
\frac{1}{2}\Delta v\geq&\frac{v^2+2ahv+2(b-N)v}{n}-\left\langle \nabla h, \nabla v\right\rangle-av-R_cv\\
&-\frac{v}{2}-\frac{\left|h \right|v}{2}-\frac{\left|h \right|\left|\nabla a \right|^2}{2}-\frac{\left|\nabla b \right|^2}{2}-\left| a\right| N.
\end{split}
\end{equation}
Next, for any $l\geq0$ and $\eta\in C^\infty_0(B_1)$, we multiply by $\eta^2v^l$ on both sides of (\ref{2.9}), then
\begin{equation}\label{2.10}
\begin{split}
\int\frac{1}{2}\eta^2v^l\Delta v
&\geq\int\eta^2v^l\left\lbrace \frac{v^2+2ahv+2(b-N)v}{n}-\left\langle \nabla h, \nabla v\right\rangle-av-\frac{\left|h \right|v}{2}-\frac{\left|h \right|\left|\nabla a \right|^2}{2}\right\rbrace\\
&\quad-\int\eta^2v^l \left\lbrace\frac{\left|\nabla b \right|^2}{2}+\frac{v}{2}+\left| a\right| N+R_cv\right\rbrace.
\end{split}
\end{equation}
By Green formula, we have
$$\int\eta^2v^{l+1}\Delta h=-\int \left\langle \nabla\left( \eta^2v^{l+1}, \nabla h\right) \right\rangle =-\int (l+1)\eta^2v^l\left\langle \nabla v, \nabla h\right\rangle-\int 2\eta v^{l+1}\left\langle \nabla\eta, \nabla h\right\rangle,$$
so we derive from (\ref{2.6}) and Cauchy-Schwartz inequality that
\begin{align}
\int\eta^2v^l\left\langle \nabla v, \nabla h\right\rangle&=-\frac{1}{l+1}\int\eta^2v^{l+1}\Delta h-\frac{2}{l+1}\int v^{l+1}\left\langle \nabla\eta, \eta\nabla h\right\rangle\nonumber\\
&\leq -\frac{1}{l+1}\int\eta^2v^{l+1}\left( N-v-ah-b\right)+\frac{1}{l+1}\int v^{l+1}\left(\left|\nabla\eta \right|^2 +\eta^2\left|\nabla h \right|^2\right)\nonumber\\
&\leq \frac{2}{l+1}\int\eta^2v^{l+2}+\frac{1}{l+1}\int v^{l+1}\left|\nabla\eta \right|^2+\frac{1}{l+1}\int\left| a\right| \left| h\right|\eta^2v^{l+1}.\label{2.11}
\end{align}
For the second inequality, we use the facts that $\left|\nabla h \right|^2=v-N$ and $2N-b>0$.

Since $v\geq1$, we have $v^{l+1}\geq v^l$. Then by combining (\ref{2.10}) and (\ref{2.11}), there holds
\begin{equation}\label{2.12}
\begin{split}
\int\frac{1}{2}\eta^2v^l\Delta v\geq&\int\eta^2v^{l+2}\left(\frac{1}{n}-\frac{2}{l+1} \right)-\int\left| a\right| \left| h\right| \eta^2v^{l+1}\left(\frac{2}{n}-\frac{1}{l+1} \right)-\frac{1}{l+1}\int v^{l+1}\left|\nabla\eta \right|^2\\
&+\frac{2}{n}\int\eta^2v^{l+1}\left( b-N\right)-\int\frac{\left|h \right|\left|\nabla a \right|^2}{2}\eta^2v^{l+1}-\int R_c\eta^2v^{l+1}\\
&-\int\left( \left| a\right| +\frac{1}{2}+\frac{\left|\nabla b \right|^2}{2}+\left| a\right|N\right)\eta^2v^{l+1} .
\end{split}
\end{equation}
Let $l\geq 2n-1$, then $\frac{1}{n}-\frac{2}{l+1}\geq0$ and $0<\frac{2}{n}-\frac{1}{l+1}\leq1$. As a consequence, we deduce from (\ref{2.12}) that
\begin{equation}\label{2.13}
\begin{split}
\int\frac{1}{2}\eta^2v^l\Delta v\geq&-\sup\limits_{B_1}\left| a\right|\int \left| h\right| \eta^2v^{l+1}-\frac{1}{l+1}\int v^{l+1}\left|\nabla\eta \right|^2-\sup\limits_{B_1}\left(N-b\right)\frac{2}{n}\int\eta^2v^{l+1}\\
&-\sup\limits_{B_1}\frac{\left|\nabla a \right|^2}{2}\int\left|h \right|\eta^2v^{l+1}-\int R_c\eta^2v^{l+1}\\
&-\sup\limits_{B_1}\left( \left| a\right| +\frac{1}{2}+\frac{\left|\nabla b \right|^2}{2}+\left| a\right|N\right)\int\eta^2v^{l+1}.
\end{split}
\end{equation}
By (5.11) of \cite{DWZ}, making use of integration by parts, we have the following general integral inequality
\begin{equation}\label{2.14}
\int\left| \nabla\left( \eta v^{\frac{l+1}{2}}\right) \right|^2\leq-\frac{\left( l+1\right)^2 }{2l}\int\eta^2v^l\Delta v+\frac{\left( l+1\right)^2+l }{l^2} \int v^{l+1}\left| \nabla\eta\right|^2-\frac{l+1}{l}\int \eta v^{l+1}\Delta \eta.
\end{equation}
Then by (\ref{2.14}), we infer from (\ref{2.13}) that there exists a constant $C_1=C_1(D_1,D_2,D_3,D_4)$ such that
\begin{equation}\label{2.15}
\begin{split}
\int\left| \nabla\left( \eta v^{\frac{l+1}{2}}\right) \right|^2\leq&\frac{l+1}{l}\int v^{l+1}\left| \nabla\eta\right|^2+\frac{(l+1)^2C_1}{l}\int\left| h\right| \eta^2v^{l+1}+\int R_c\eta^2v^{l+1}\\
&+\left\lbrace \frac{(l+1)^2}{l}+\frac{2(l+1)^2}{nl}\right\rbrace C_1\int \eta v^{l+1}\\
&+\frac{\left( l+1\right)^2+l }{l^2} \int v^{l+1}\left| \nabla\eta\right|^2-\frac{l+1}{l}\int \eta v^{l+1}\Delta \eta.
\end{split}
\end{equation}

Now, we need to control $\Delta\eta$. To this end, for $0<r<1$, let $\phi\in C^\infty_0\left( [ 0, +\infty)\right) $ be a cut-off function such that
$$0\leq\phi\leq1, \quad\hspace*{0.3em}\phi(t)\equiv1\hspace*{0.3em} \mbox{for}\hspace*{0.3em}0\leq t\leq r,\quad\hspace*{0.3em}\phi(t)\equiv0\hspace*{0.3em}\mbox{for}\hspace*{0.3em}t\geq1\hspace*{0.3em}\quad\mbox{and}\quad\hspace*{0.3em}\phi'\leq0.$$
Let $\eta(y)=\phi(d(x,y))$, then $\left| \nabla\eta\right|=\left| \phi'\right|$ and for $\psi(y)=\left(\Delta d-\frac{n-1}{d} \right)_+ $ there holds
\begin{equation}\label{2.16}
\begin{split}
\Delta\eta&=\phi''+\phi'\Delta d=\phi''+\phi'\left( \Delta d-\frac{n-1}{d}+\frac{n-1}{d}\right)\\
&\geq\phi''+\phi'\left( \psi+\frac{n-1}{d}\right)\geq-\left|\phi'' \right|-\frac{(n-1)\left|\phi' \right| }{r}-\left| \phi'\right|\psi.
\end{split}
\end{equation}
By the well-known Calabi's trick, without loss of generality, we can assume $\eta\in C^\infty_0(B_1)$. Combining (\ref{2.15}) and (\ref{2.16}), we know there exists a constant $C_2=C_2(C_1,n)$ such that
\begin{equation}\label{2.17}
\begin{split}
&\int\left| \nabla\left( \eta v^{\frac{l+1}{2}}\right) \right|^2\\
&\leq C_2l\int\left[ \left( \left| \phi''\right| +\frac{\left|\phi' \right| }{r}+\left|  \phi'\right|\psi \right) \eta v^{l+1}+\left| \phi'\right|^2v^{l+1} +R_c\eta^2v^{l+1}+\left| h\right|\eta^2v^{l+1}+\eta^2v^{l+1}\right].
\end{split}
\end{equation}Theorem \ref{th1.1} tells us that there exists a small enough positive number $k(n, p)$ such that (\ref{1.4}) holds true if $\kappa<k(n, p)$. For $n\geq3$, we denote $\frac{n}{n-2}>1$ by $\mu$ temporarily.
Then by (\ref{1.4}), i.e. letting $f=\eta v^{\frac{l+1}{2}}$ and $r=1$, we have
\begin{equation}\label{2.18}
\begin{split}
&\left( \fint\left( \eta^2v^{l+1}\right)^\mu \right)^{\frac{1}{\mu}}\\
&\leq C_S^2C_2l\fint\left[ \left( \left| \phi''\right| +\frac{\left|\phi' \right| }{r}+\left|  \phi'\right|\psi \right) \eta v^{l+1}+\left| \phi'\right|^2v^{l+1}+R_c\eta^2v^{l+1}+\left| h\right|\eta^2v^{l+1}+\eta^2v^{l+1}\right].
\end{split}
\end{equation}
For $n=2$, applying (\ref{1.3}) to $f^3$ yields
\begin{equation*}
	\begin{split}
		\left( \fint f^6\right)^{\frac{1}{2}}\leq C_S\fint \left|\nabla f^3 \right|=3C_S\fint \left|f^2\nabla f \right|&\leq3C_S\left( \fint \left|f^4 \right|\right)^{\frac{1}{2}}\left( \fint \left| \nabla f \right|^2\right)^{\frac{1}{2}}\\
		&\leq3C_S\left( \fint \left|f^6 \right|\right)^{\frac{1}{3}}\left( \fint \left| \nabla f \right|^2\right)^{\frac{1}{2}},
	\end{split}
\end{equation*}
for the last inequality, we use the monotone inequality $\left| \left| f\right| \right|^*_4\leq\left| \left| f\right| \right|^*_6$.
So by (\ref{2.17}) we also have
\begin{equation}\label{2.19}
	\begin{split}
		&\left( \fint \left( \eta^2v^{l+1}\right) ^3\right) ^{\frac{1}{3}}\\
		&\leq C_S^2C_2l\fint\left[ \left( \left| \phi''\right| +\frac{\left|\phi' \right| }{r}+\left|  \phi'\right|\psi \right) \eta v^{l+1}+\left| \phi'\right|^2v^{l+1}+R_c\eta^2v^{l+1}+\left| h\right|\eta^2v^{l+1}+\eta^2v^{l+1}\right],
	\end{split}
\end{equation}
i.e., when $n=2$, we can choose $\mu=3$ for (\ref{2.18}). But we need to point out that $C_2$ in (\ref{2.19}) is 9 times larger than $C_2$ in (\ref{2.18}).

To make our expressions more concise, from now on, we define
\begin{equation*}
n'=
\left\{
\begin{array}{lr}
n ,&n\geq3,\\
3 ,&n=2;
\end{array}\quad and \quad
\right.	
\mu=
\left\{
\begin{array}{lr}
\frac{n}{n-2} ,&n\geq3,\\
3 ,&n=2.
\end{array}
\right.
\end{equation*}	
Next, by H\"{o}lder's inequality and Young's inequality, it holds that
\begin{equation}\label{2.20}
\begin{split}
\fint R_c\eta^2v^{l+1}&\leq\left| \left| R_c\right| \right|^*_p\left(\fint\left( \eta^2v^{l+1}\right)^{\frac{p}{p-1}}  \right)^{\frac{p-1}{p}}\\
&\leq\left| \left| R_c\right| \right|^*_p\left(\fint  \eta^2v^{l+1}\right)^{\frac{(p-1)a}{p}}\left( \fint \left(\eta^2v^{l+1} \right)^\mu \right)^{\frac{(p-1)(1-a)}{p}}\\
&\leq\left| \left| R_c\right| \right|^*_p\left[ \varepsilon\left(\fint\left( \eta^2v^{l+1}\right)^\mu  \right)^{\frac{1}{\mu}} +\varepsilon^{-\frac{(1-a)\mu}{a}}\left(\fint \eta^2v^{l+1} \right) \right].
\end{split}
\end{equation}
Here we need to choose some $a'$ such that
$$0<a'<1\quad \mbox{and}\quad a'+(1-a')\mu=\frac{p}{p-1}.$$
Simultaneously, the above inequalities lead to
$$p>\frac{n}{2}\quad\mbox{for}\hspace*{0.3em}n\geq3\quad\mbox{and}\quad p>\frac{3}{2}\quad\mbox{for}\hspace*{0.3em}n=2.$$
Let $\varepsilon=\left( 5C_S^2C_2l\left| \left| R_c\right| \right|^*_p\right)^{-1}$, we can see easily that there exists a positive constant $C_3=C_3(C_2,n',p)$ such that
\begin{equation}\label{2.21}
C_S^2C_2l\fint R_c\eta^2v^{l+1}\leq\frac{1}{5}\left(\fint\left( \eta^2v^{l+1}\right)^\mu\right)^{\frac{1}{\mu}}+C_3\left(C_S^2l \left| \left| R_c\right| \right|^*_p\right)^{\frac{2p}{2p-n'}}\fint\eta^2v^{l+1}.
\end{equation}
Similarly, we have
\begin{equation}\label{2.22}
C_S^2C_2l\fint \left| h\right| \eta^2v^{l+1}\leq\frac{1}{5}\left(\fint\left( \eta^2v^{l+1}\right)^\mu  \right)^{\frac{1}{\mu}}+C_3\left(C_S^2l \left| \left| \left|h \right| \right| \right|^*_p\right)^{\frac{2p}{2p-n'}}\fint\eta^2v^{l+1},
\end{equation}
and
\begin{equation}\label{2.23}
C_S^2C_2l\fint \eta^2v^{l+1}\leq\frac{1}{5}\left(\fint\left( \eta^2v^{l+1}\right)^\mu  \right)^{\frac{1}{\mu}}+C_3\left(C_S^2l \right)^{\frac{2p}{2p-n'}}\fint\eta^2v^{l+1}.
\end{equation}

For the term with $\psi$, denoting $\frac{(n-1)(2p-1)}{2p-n}$ by $C(n, p)$ and using the H\"{o}lder inequality and the Laplacian comparison estimate (\ref{2.1}), we obtain
\begin{equation}\label{2.24}
\begin{split}
\fint \psi\eta\left| \phi'\right|v^{l+1}&\leq\left|\left|  \psi \right| \right|^*_{2p}\left| \left| \eta\phi'v^{l+1}\right| \right|^*_{\frac{2p}{2p-1}} \\
&\leq C(n, p)^{\frac{1}{2}}\left( \left| \left| R_c\right| \right|^*_p\right)^{\frac{1}{2}} \left| \left| \eta\phi'v^{l+1}\right| \right|^*_{\frac{2p}{2p-1}}.
\end{split}
\end{equation}
Since $p>\frac{n'}{2}$, then $0<b'=\frac{p}{\mu(2p-1)}<1$, for some $\tilde{\varepsilon}>0$ we can derive
\begin{align}
\left| \left| \eta\phi'v^{l+1}\right| \right|^*_{\frac{2p}{2p-1}}&=\left[ \fint \left( \eta^2v^{l+1}\right)^{b'\mu} \left( \left| \phi'\right|^2v^{l+1} \right)^{\frac{p}{2p-1}} \right]^{\frac{2p-1}{2p}}\nonumber\\
&\leq\left[ \left( \fint\left( \eta^2v^{l+1}\right)^\mu \right)^{b'}\left( \fint \left( \left| \phi'\right|^2v^{l+1}\right)^\frac{n'p}{n'p+2p-n'} \right)^{\frac{n'p+2p-n'}{n'(2p-1)}} \right]^\frac{2p-1}{2p}\nonumber\\\
&\leq\left[ \left( \fint\left( \eta^2v^{l+1}\right)^\mu \right)^{b'}\left( \fint  \left| \phi'\right|^2v^{l+1}\right)^{\frac{p}{2p-1}} \right]^\frac{2p-1}{2p}\nonumber\\
&\leq\tilde{\varepsilon}\left( \fint\left( \eta^2v^{l+1}\right)^\mu \right)^\frac{1}{\mu}+\frac{1}{4\tilde{\varepsilon}}\fint\left| \phi'\right|^2v^{l+1}\label{3.21*}
\end{align}
Here, we have used the monotonicity in the second inequality
$$\left| \left|\left| \phi'\right|^2v^{l+1} \right| \right|^*_{\frac{n'p}{n'p+2p-n'}}\leq\left| \left| \left| \phi'\right|^2v^{l+1}\right| \right|^*_1$$
since
$$0<\frac{n'p}{n'p+2p-n'}<1,$$
and we have also used Young's inequality and the fact $\frac{b'(2p-1)}{p}=\frac{1}{2\mu}$ in the above third inequality. Now we set
$$\tilde{\varepsilon}=\left( 5 C_2C_S^2lC(n, p)^{\frac{1}{2}}\left( \left| \left| R\right| \right|^*_p\right)^{\frac{1}{2}}\right)^{-1},$$
then
\begin{equation}\label{2.25}
C_S^2C_2l\fint \psi\eta\left| \phi'\right|v^{l+1}\\\leq\frac{1}{5}\left(\fint\left( \eta^2v^{l+1}\right)^\mu \right)^{\frac{1}{\mu}}+\frac{5}{4}\left( C_S^2C_2l\right)^2C(n, p)\left| \left| R_c\right| \right|^*_p\fint \left| \phi'\right|^2v^{l+1}.
\end{equation}
Substituting (\ref{2.21}), (\ref{2.22}), (\ref{2.23}) and (\ref{2.25}) into (\ref{2.18}), then by Remark \ref{r1.2} we can see that there holds
\begin{align}\label{2.26}
&\left(\fint\left( \eta^2v^{l+1}\right)^\mu  \right)^{\frac{1}{\mu}}\nonumber\\
&\leq5C_S^2C_2l\left[ \fint\left( \left| \phi''\right| +\frac{\left|\phi' \right| }{r} \right) \eta v^{l+1}+\left( 1+\frac{5}{4}C_S^2C_2lC(n, p)\left| \left| R_c\right| \right|^*_p\right)\fint \left| \phi'\right|^2v^{l+1}\right]\nonumber\\
&\hspace*{1em}+5C_3\left( C_S^2l\right)^{\frac{2p}{2p-n'}}\left( 1+\left(\left| \left| R_c\right| \right|^*_p  \right)^\frac{2p}{2p-n'} +\left(\left| \left| h\right| \right|^*_p  \right)^\frac{2p}{2p-n'} \right)\fint\eta^2v^{l+1}\nonumber\\
&\leq5C_S^2C_2l\left[ \fint\left( \left| \phi''\right| +\frac{\left|\phi' \right| }{r} \right) \eta v^{l+1}+\left( 1+\frac{5}{4}C_S^2C_2lC(n, p)2^{\frac{1}{p}}\kappa(p,1)\right)\fint \left| \phi'\right|^2v^{l+1}\right]\nonumber\\
&\hspace*{1em}+5C_3\left( C_S^2l\right)^{\frac{2p}{2p-n'}}\left( 1+\left(2^{\frac{1}{p}}\kappa(p,1)\right)^\frac{2p}{2p-n'} +\left(\left| \left| h\right| \right|^*_p  \right)^\frac{2p}{2p-n'} \right)\fint\eta^2v^{l+1}
\end{align}

For integers $j\geq0$, we define $$r_j= \theta+\delta-\delta\left( \sum^j\limits_{i=0}2^{-i-1}\right) $$
where $\theta,\delta$ are two constants such that $0<\delta\leq\frac{1}{2}$ and $\frac{1}{2}\leq\theta\leq\frac{4}{5}-\delta$, then $\theta<r_j<\theta+\delta$. Moreover, we can choose $\eta_j(y)=\phi_j(d(y))\in C^\infty_0(B_{r_{j-1}})$ such that
$$\phi_j\equiv1\quad\mbox{on}\hspace*{0.3em}B_{r_j},\quad \hspace*{0.3em}\phi_j\equiv0\quad\mbox{on}\hspace*{0.3em}B_1\setminus B_{r_{j-1}},\,\,B_{r_{-1}}=B_{\theta+\delta}$$
and
$$\left|\phi_j'\right|\leq\frac{2^{j+3}}{\delta},\quad\quad \left| \phi_j''\right|\leq\frac{2^{2j+6}}{\delta^2}.$$
In this situation, we have $$\frac{\left|\phi_j'\right|}{r_j}\leq\frac{2^{2j+6}}{\delta^2},\quad\quad  \left|\phi_j'\right|^2\leq\frac{2^{2j+6}}{\delta^2}\quad and
\quad \eta_{j+1}\leq\eta^\mu_j\leq\eta_j$$ and especially
\begin{equation}\label{2.27}
\eta_j\equiv1
\end{equation}
on the  supports of $\eta_{j+1}$ and $\phi_{j+1}$.

It is easy to know that there exists a minimal integer $m$ such that $\mu^m\geq2n$. Set
$$q=\max{\left\lbrace 2,\frac{2p}{2p-n'}\right\rbrace }.$$
For simplicity, we denote
$$A\equiv \delta^{-2}\left[ 5C_S^2C_2\left(3+\frac{5}{4}C_S^2C_2C(n, p)2^{\frac{1}{p}}\kappa(p,1)\right)+5C_3\left( C_S^2\right)^{\frac{2p}{2p-n'}}\left( 1+\left(2^{\frac{1}{p}}\kappa(p,1)\right)^\frac{2p}{2p-n'} +\left(\left| \left| h\right| \right|^*_p  \right)^\frac{2p}{2p-n'} \right)\right]$$ and
$$B\equiv5C_S^2C_2\left(3+\frac{5}{4}C_S^2C_2C(n, p)2^{\frac{1}{p}}\kappa(p,1)\right)+5C_3\left( C_S^2\right)^{\frac{2p}{2p-n'}}\left( 1+\left(2^{\frac{1}{p}}\kappa(p,1)\right)^\frac{2p}{2p-n'} +\left(\left| \left| h\right| \right|^*_p  \right)^\frac{2p}{2p-n'} \right).$$
Then, according to (\ref{2.26}), it holds that
\begin{equation}\label{2.28}
\left(\fint\left( \eta^2_{j+1}v^{l+1}\right)^\mu\right)^{\frac{1}{\mu}}\leq Al^q2^{2j+6}\fint \eta^{2\mu}_jv^{l+1}
\leq A(l+1)^q2^{2j+6}\fint \eta^{2\mu}_jv^{l+1}.
\end{equation}
Now, (\ref{2.26}) is valid for all $l+1\geq\mu^m$, i.e. for $l+1=\mu^m,\,\mu^{m+1},\,\mu^{m+2},\cdots\cdots$. By (\ref{2.27}) and (\ref{2.28}) we can infer
\begin{equation}\label{2.29}
\left( \fint \eta_1^{2\mu}v^{\mu^{m+1}}\right\rbrace ^{\frac{1}{\mu}}\leq A\mu^{mq}2^{6}\fint\eta^2_0v^{\mu^m},
\end{equation}
and
\begin{equation*}
\begin{split}
\left(\fint\eta_2^{2\mu}v^{\mu^{m+2}}\right)^{\frac{1}{\mu}}\leq A\mu^{(m+1)q}2^{8}\fint \eta_1^{2\mu}v^{\mu^{m+1}}
&\leq A\mu^{(m+1)q}2^{8}\left( A\mu^{mq}2^{6}\fint\eta^2_0v^{\mu^m}\right)^\mu\\
&=A^{1+\mu}\mu^{\mu mq+(m+1)q}2^{6\mu+8}\left( \fint\eta^2_0v^{\mu^m}\right)^\mu,
\end{split}
\end{equation*}
i.e.
\begin{equation}\label{2.30}
\left(\fint\eta_2^{2\mu}v^{\mu^{m+2}}\right)^{\frac{1}{\mu^2}}\leq A^{\frac{1+\mu}{\mu}}\mu^{\frac{\mu mq+(m+1)q}{\mu}}2^{\frac{6\mu+8}{\mu}}\fint\eta^2_0v^{\mu^m}.
\end{equation}
Furthermore, by repeating the above iteration step by step we can infer that for any positive integer $j$ there holds true
\begin{equation}\label{2.31}
\left(\fint\eta_{j+1}^{2\mu}v^{\mu^{m+j+1}}\right)^{\frac{1}{\mu^{j+1}}}\leq A^{\frac{R_j}{\mu^j}}\mu^{\frac{S_j}{\mu^j}}2^{\frac{T_j}{\mu^j}}\fint\eta^2_0v^{\mu^m}
\end{equation}
where
$$R_j=\sum^j\limits_{i=0}\mu^i,$$
$$S_0=mq,\quad S_j=\mu S_{j-1}+(m+j)q,$$
$$T_0=6,\quad T_j=\mu T_{j-1}+6+2j.$$
Letting $j\rightarrow+\infty$, then it's easy to see that $\frac{R_j}{\mu^j}\rightarrow\frac{n'}{2}$. As for $S_j$, there holds
$$0<\frac{S_{j+1}}{\mu^{j+1}}-\frac{S_j}{\mu^j}=\frac{(m+j)q}{\mu^j}\leq\frac{2jq}{\mu^j}\quad\quad \mbox{for}\,\,j\geq m,$$
so we conclude that $\frac{S_j}{\mu^j}$ converges to a number which is only related to $\mu$, $m$ and $q$. Similarly, $\frac{T_j}{\mu^j}$ also converges to a real number which is only related to $\mu$ and $q$. In other words, we derive that there exists a constant $C_4=C_4(\mu, n', q)$ such that
\begin{equation*}
\left| \left| v\right| \right|_{\infty,B_{\theta}}=\left| \left| v\right| \right|^*_{\infty,B_{\theta}}\leq C_4^{\frac{1}{\mu^m}}A^\frac{n'}{2\mu^m}\left| \left| v\right| \right|^*_{\mu^m,B_{\theta+\delta}}.
\end{equation*}
Then, by Lemma \ref{l2.3} we let $K_1=C_4B^{\frac{n'}{2}},\,K_2=n',\,t=\mu^m,\, s=1,\, R=1,\,\tau=\frac{3}{10}$ to obtain further that there exists a constant $C_5=C_5\left( C_4B^\frac{n'}{2},\mu,n',m\right) $ such that
\begin{equation}\label{2.32}
\left| \left| v\right| \right|_{\infty,B_{\frac{1}{2}}}\leq C_5\left| \left| v\right| \right|^*_{1,B_\frac{4}{5}}.
\end{equation}

Next, we choose some $\eta\in C^\infty_0(B_1)$ such that
$$\eta\equiv1\quad\mbox{on}\hspace*{0.3em}B_\frac{4}{5}\quad\mbox{and}\quad \left| \nabla\eta\right|\leq10.$$
For $\eta^2v$, there holds
\begin{equation}\label{2.33}
\begin{split}
\int\eta^2v=-\int\eta^2\left( \Delta h+ah+b-N\right).
\end{split}
\end{equation}
By Green's formula, we have
$$-\int\eta^2\Delta h=2\int\left\langle \nabla\eta, \eta\nabla h\right\rangle\leq\frac{1}{2}\int\eta^2v+2\int\left| \nabla\eta\right|^2\leq\frac{1}{2}\int\eta^2v+200\left| B_1\right| .$$
There also hold
$$-\int\eta^2ah\leq\sup\limits_{B_1}\left| a\right|\int\eta^2\left| h\right|\leq\sup\limits_{B_1}\left| a\right|\left| \left| h\right| \right|^*_1\left|B_1 \right|\leq\sup\limits_{B_1}\left| a\right|\left| \left| h\right| \right|^*_p\left|B_1 \right|,$$
and
$$\int\eta^2(N-b)\leq\left( 2\sup\limits_{B_1} \left| b\right|+1\right)\left|B_1 \right| $$
since $N=\sup\limits_{B_1}\left| b\right|+1$.
Then, in view of the above estimates we deduce from (\ref{2.33}) that there holds true
\begin{equation}\label{2.34}
\left| \left| v\right| \right|^*_{1,B_\frac{4}{5}}\leq\frac{\int\eta^2v}{\left|B_{\frac{4}{5}}\right|}\leq\frac{\left( 400+2\sup\limits_{B_1}\left| a\right|\left| \left| h\right| \right|^*_p+4\sup\limits_{B_1}\left|  b\right|+2\right)\left|B_1 \right|}{\left|B_{\frac{4}{5}} \right|}.
\end{equation}
Moreover, by Lemma \ref{l2.2} there exists some $k(n, p)$ small enough such that (\ref{2.2}) holds true if $\kappa<k(n, p)$. Hence, we have
$$\left|B_1 \right|\leq\frac{2\cdot5^n}{4^n}\left|B_{\frac{4}{5}}\right|.$$
Therefore, by substituting (\ref{2.34}) into (\ref{2.32}) we accomplish the entire proof of Theorem \ref{th1.2} provided $k$ is small enough such that Theorem \ref{th1.1} and Lemma \ref{l2.2} are valid.

Furthermore, if $a\equiv0$, then from (\ref{2.13}) and (\ref{2.33}), we know the terms which contain $h$ vanish, so $C$ is not relevant to $D_1,\,D_3$ and $D_5$. When $a\geq0$ and $b\geq0$ on $B_r$, we take $N=b+1$, then it follows that $aN\geq0$. Therefore, from (\ref{2.6}), (\ref{2.8}), (\ref{2.13}) and (\ref{2.33}) we know the terms which contain $b,\,\left| \nabla b\right| $ or $N$ vanish, so $C$ is not relevant to $D_2$ and $D_4$.
\end{proof}

Since the proof of Corollary \ref{c1.2'} is almost the same as Theorem \ref{th1.2}, we only need to give the outline.
\begin{proof}[\textbf{Proof of Corollary \ref{c1.2'}}]
	Let $v=\left| \nabla h\right|^2+ah$, $B'=\left\lbrace x\in B_1:v>1\right\rbrace$ and $\chi_{B'}$ be the characteristic function of $B'$. Since $B'$ is relative open in $B_1$, we may assume $v\chi_{B'}$ is smooth on $B_1$. Let $v'$ denote $v\chi_{B'}$.
	In this situation, we have
	\begin{equation}\label{2.35}
		\left( \Delta h+v+b\right)\chi_{B'}=0,
	\end{equation}
	then
	\begin{equation}\label{2.36}
		\frac{1}{2}\Delta v'=\frac{\chi_{B'}}{2}\Delta\left( \left| \nabla h\right|^2+ah\right)=\frac{\chi_{B'}}{2}\left(\Delta \left| \nabla h\right|^2+a\Delta h+h\Delta a+2\left\langle\nabla a, \nabla h \right\rangle \right).
	\end{equation}
By Bochner Formula, we have
\begin{align}\label{2.37}
\frac{1}{2}\Delta \left| \nabla h\right|^2&=\left| D^2h\right|^2+\left\langle \nabla h, \nabla \Delta h\right\rangle+Ric\left( \nabla h, \nabla h\right)\nonumber\\
&\geq\frac{1}{n}\left( \Delta h\right)^2+\left\langle \nabla h, \nabla (-v-b)\right\rangle-R_c\left( v-ah\right)\nonumber\\	&=\frac{(v+b)^2}{n}-\left\langle \nabla h, \nabla v\right\rangle-\left\langle \nabla h, \nabla b\right\rangle-R_cv+ahR_c.
\end{align}
	By Cauchy-Schwartz inequality, there hold
	\begin{equation}\label{2.38}
		\left\langle \nabla a, \nabla h\right\rangle\leq\frac{\left|\nabla a \right|^2+\left|\nabla h \right|^2}{2}=\frac{\left|\nabla a \right|^2+v-ah}{2}
	\end{equation}
	and
	\begin{equation}\label{2.39}
		\left\langle \nabla b, \nabla h\right\rangle\leq\frac{\left|\nabla b \right|^2+\left|\nabla h \right|^2}{2}=\frac{\left|\nabla b \right|^2+v-ah}{2}.
	\end{equation}
	Plugging (\ref{2.37}), (\ref{2.38}) and (\ref{2.39}) into (\ref{2.36}), we have
\begin{align}\label{2.40}
\frac{1}{2}\Delta v'\geq&\chi_{B'}\left( \frac{(v+b)^2}{n}-\left\langle \nabla h, \nabla v\right\rangle-\frac{\left|\nabla b \right|^2+v-ah}{2}-R_cv+ ahR_c\right)\nonumber\\
&+\chi_{B'}\left( \frac{a}{2}(-v-b)+\frac{h\Delta a}{2}-\frac{\left|\nabla a \right|^2+v-ah}{2}\right)\nonumber\\
\geq&\chi_{B'}\left( \frac{(v+b)^2}{n}-\left\langle \nabla h, \nabla v\right\rangle-v+ah-R_cv+ ahR_c\right)\nonumber\\
&+\chi_{B'}\left( \frac{a}{2}(-v-b)+\frac{h\Delta a}{2}-\frac{\left|\nabla a \right|^2+\left|\nabla b \right|^2}{2}\right)\nonumber\\
\geq&\chi_{B'}\left( \frac{(v+b)^2}{n}-\left\langle \nabla h, \nabla v\right\rangle-v-\left| ah\right| -R_cv+ \left| ah\right| R_c\right)\nonumber\\
&+\chi_{B'}\left( -\frac{\left| a\right| }{2}(v+b)+\frac{\left|h \Delta a\right| }{2}-\frac{\left|\nabla a \right|^2+\left|\nabla b \right|^2}{2}\right).
\end{align}
Then we can estimate (\ref{2.40}) just like to estimate (\ref{2.9}) and the only difference is that we need to deal with the term $\left| ahR\right|$ in (\ref{2.40}). To this end, we need only to note that by H\"{o}lder's inequality there holds true
	\begin{equation}\label{2.41}
		\left| \left| aR_ch\right| \right|_{p,B_1}^*\leq\sup\limits_{B_1}\left|a \right|  \left| \left| R_c\right| \right|_{2p,B_1}^*\left| \left| h\right| \right|_{2p,B_1}^*.
	\end{equation}
	i.e., we can obtain a similar integral inequality for $R_ch\eta^2v^{l+1}\chi_{B'}$ with (\ref{2.20}). Next, for $\eta\in C^\infty_0(B_1)$ with $\eta\equiv1$ on $B_{\frac{4}{5}}$ and $\left| \nabla\eta\right|\leq10$  on $B_1$, there holds
	\begin{equation*}
	\begin{split}
	\int\eta^2v'&=-\int\eta^2\left( \Delta h+b\right)\chi_{B'}\\
	&=\left( \int2\eta\left\langle \nabla\eta, \nabla h\right\rangle-\int \eta^2b\right) \chi_{B'}\\
	&\leq\left( \int2\left| \nabla\eta\right|^2+\int\frac{1}{2}\eta^2\left| \nabla h\right|^2+\int\left| b\right|\right) \chi_{B'} \\
	&\leq\left( \int2\left| \nabla\eta\right|^2+\int\frac{1}{2}\eta^2\left( \left| \nabla h\right|^2+ah\right)-\int\frac{1}{2}\eta^2ah+\int\left| b\right|\right) \chi_{B'}\\
	&\leq\left( \int2\left| \nabla\eta\right|^2+\int\frac{1}{2}\eta^2v+\frac{\left|a \right| }{2}\int \left| h\right| +\int\left| b\right|\right) \chi_{B'},
	\end{split}
	\end{equation*}
	hence $$\left\| v'\right\|^*_{B_{\frac{4}{5}}}\leq\frac{\int4\left| \nabla\eta\right|^2+\left|a \right|\int \left| h\right|+2\int\left| b\right|}{B_{\frac{4}{5}} }\leq\left(400+2\sup\limits_{B_1}\left| a\right|\left| \left| h\right| \right|^*_{2p,B_1}+2\sup\limits_{B_1}\left| b\right|  \right)\frac{B_1}{B_{\frac{4}{5}}}.$$
	In a word, after obtaining (\ref{2.40}), we can estimate $v'$ as in the proof of Theorem \ref{th1.2}, then we accomplish the proof.
\end{proof}
As for Corollary \ref{c1.1}, its proof is also almost contained in proof of Theorem \ref{th1.2}, so we only need to point out the differences.
\begin{proof}[\textbf{Proof of Corollary \ref{c1.1}}]
Let $$h=\log\frac{u}{D},\quad N=\sup\limits_{B_1}\left| a\log D+b\right|+1\quad\mbox{and}\quad v=\left| \nabla h\right|^2+ah+N,$$ then it follows that there hold true $ah\geq0$ and $v\geq1$. In this situation, (\ref{2.6}) becomes
\begin{equation}\label{3.38'}
\Delta h=N-v-b-a\log D.
\end{equation}
Since $ahR_c\geq0,\,R_cN\geq0$ and $(N-a\log D+b)^2\geq0$, then from (\ref{2.9}) we have
\begin{equation}\label{3.39'}
\begin{split}
\frac{1}{2}\Delta v\geq&\frac{v^2+2(a\log D+b-N)v}{n}-\left\langle \nabla h, \nabla \Delta h\right\rangle\\
&-\frac{a}{2}v-\frac{a}{2}(N-b-a\log D)-R_cv.
\end{split}
\end{equation}
From the inequality $$\left\langle \nabla b, \nabla h\right\rangle\leq\frac{\left|\nabla b \right|^2+\left|\nabla h \right|^2}{2}=\frac{\left|\nabla b \right|^2+v-ah}{2}\leq\frac{\left|\nabla b \right|^2+v}{2},$$
the inequality (\ref{3.39'}) can be rewritten as
\begin{equation}\label{3.39''}
\begin{split}
\frac{1}{2}\Delta v\geq&\frac{v^2+2(a\log D+b-N)v}{n}-\left\langle \nabla h, \nabla v\right\rangle-\frac{\left| \nabla b\right|^2 }{2}\\
&-\left( \frac{a}{2}+\frac{1}{2}\right) v-\frac{a}{2}(N-b-a\log D)-R_cv.
\end{split}
\end{equation}
Since $ah\geq0$, $\left|\nabla h \right|^2=v-ah-N$ and $2N-b-a\log D\geq0$,
then, from (\ref{2.11}) we can obtain
\begin{align}\label{3.40'}
\int\eta^2v^l\left\langle \nabla v, \nabla h\right\rangle =&-\frac{1}{l+1}\int\eta^2v^{l+1}\Delta h-\frac{2}{l+1}\int v^{l+1}\left\langle \nabla\eta, \eta\nabla h\right\rangle\nonumber\\
\leq & -\frac{1}{l+1}\int\eta^2v^{l+1}\left( N-v-b-a\log D\right)\nonumber\\
&+\frac{1}{l+1}\int v^{l+1}\left(\left|\nabla\eta \right|^2 +\eta^2\left|\nabla h \right|^2\right)\nonumber\\
\leq &\frac{2}{l+1}\int\eta^2v^{l+2}+\frac{1}{l+1}\int v^{l+1}\left|\nabla\eta \right|^2.
\end{align}
As for (\ref{2.33}), since $v=\Delta h+a\log D+b-N$ and $ah+N\geq0$, then for $\eta\in C^\infty_0(B_1)$ with $\eta\equiv1$ on $B_{\frac{4}{5}}$ and $\left| \nabla\eta\right|\leq10$  on $B_1$, we have
\begin{equation*}
\begin{split}
\int\eta^2v&=-\int\eta^2\left( \Delta h+a\log D+b-N\right)\\
&=\int2\eta\left\langle \nabla\eta, \nabla h\right\rangle+\int \eta^2\left( N-b-a\log D\right)\\
&\leq\int2\left| \nabla\eta\right|^2+\frac{1}{2}\eta^2\left| \nabla h\right|^2+\int \eta^2\left( N-b-a\log D\right)\\
&\leq\int2\left| \nabla\eta\right|^2+\int\frac{1}{2}\eta^2\left( \left| \nabla h\right|^2+ah+N\right) +\int \eta^2\left( N-b-a\log D\right)\\
&=\int2\left| \nabla\eta\right|^2+\int\frac{1}{2}\eta^2v +\int \eta^2\left( N-b-a\log D\right),
\end{split}
\end{equation*}
hence $$\left\| v\right\|^*_{B_{\frac{4}{5}}}\leq\frac{\int4\left| \nabla\eta\right|^2+\int 2\eta^2\left( N-b-a\log D\right)}{B_{\frac{4}{5}}}\leq\left( 400+4N\right)\frac{B_1}{B_{\frac{4}{5}}}.$$
Now, all the terms which contain $h$ vanish, hence we can iterate as in the proof of Theorem \ref{th1.2} without assuming $D_5<+\infty$ to get $$\sup\limits_{B\left( x,\frac{1}{2}\right) }\left|\nabla h \right|^2\leq C.$$
Moreover, if $a\log D+b\geq0$ on $B_r$, then we can see that picking $N=a\log D+b+1$ is a proper choice. As a consequence, the term $N-a\log D-b$ vanishes in the above procedures, so the constant $C$ does not depend on $D_2$, $D_4$ or $D$.
\end{proof}

\section{\textbf{Proof of Theorem \ref{th1.2'}}}
Our main method adopted here is still the powerful Moser's iteration, but the procedures are a little different from those in the proof of Theorem \ref{th1.2}.
\begin{proof}
\noindent(1). For any $q\geq1$, we have$$\Delta u^q=qu^{q-1}\Delta u+q(q-1)u^{q-2}\left| \nabla u\right|^2\geq qu^{q-1}\Delta u,$$
then from equation (\ref{equ}), it's obvious that
\begin{equation}\label{4.1}
-\Delta u^q\leq aqu^q\log u+q(b)_+u^q.
\end{equation}
For a cut-off function $\eta\in C^\infty_0(B_1)$, we set $w=u^q$, multiply by $\eta^2w$ on both sides of (\ref{4.1}) and integrate on $B_1$ to get
\begin{equation*}
\int\left\langle \nabla(\eta^2w), \nabla w\right\rangle\leq aq\int(\eta w)^2\log u+q\int(b)_+(\eta w)^2,
\end{equation*}
and this leads to
\begin{equation}\label{4.2}
\int\left| \nabla\left( \eta w\right) \right|^2\leq\int\left| \nabla\eta\right|^2w^2+  \frac{aq}{2}\int(\eta w)^2\log u^2+q\int(b)_+(\eta w)^2.
\end{equation}
Without loss of generality, analogous to the proof of Corollary \ref{c1.2'} we can assume $u>1$ since otherwise it's proper to consider $u\chi_{B'}$ instead of $u$, where $B'=\left\lbrace x\in B_1:u(x)>1 \right\rbrace$.

Now it's easy check that for any positive number $\tilde{e}>0$, the quantity $\log u^2-u^{2\tilde{e}}$ has a upper bound $C=C(\tilde{e})$. Hence the H\"{o}lder's inequality tells us that for fixed $p>\frac{n}{2}$ there holds
\begin{align}\label{4.3}
\frac{aq}{2}\int(\eta w)^2\log u^2&\leq\frac{aq}{2}\int(\eta w)^2u^{2\tilde{e}}+\frac{C(\tilde{e})aq}{2}\int(\eta w)^2\nonumber\\
&\leq \frac{aq}{2}\left( \int u^{2\tilde{e}p}\right)^{\frac{1}{p}} \left( \int(\eta w)^{\frac{2p}{p-1}}\right)^{\frac{p-1}{p}}+\frac{C(\tilde{e})aq}{2}\int(\eta w)^2.
\end{align}
Without loss of generality, we may assume $\left\| u\right\|_{2,B_1}=1$. Let $\tilde{e}p=1$, then we have
\begin{equation}\label{4.4}
\frac{aq}{2}\int(\eta w)^2\log u^2\leq \frac{qa}{2}\left(\int (\eta w)^{\frac{2p}{p-1}}\right)^{\frac{p-1}{p}}+\frac{C(\tilde{e})aq}{2}\int(\eta w)^2.
\end{equation}
By interpolation inequality similar to (\ref{3.21*}), for any $B>0$, there exists a constant $C=C(n, p, B)$ such that for $n\geq3$
\begin{equation}\label{4.5}
\left(\fint (\eta w)^{\frac{2p}{p-1}}\right)^{\frac{p-1}{p}}\leq B\left(\fint (\eta w)^{\frac{2n}{n-2}}\right)^{\frac{n-2}{n}}+C(n, p,B)\fint(\eta w)^2,
\end{equation}
therefore
\begin{equation}\label{4.6}
\frac{aq}{2}\int(\eta w)^2\log u^2\leq \frac{qaB}{2}\left(\int (\eta w)^{\frac{2n}{n-2}}\right)^{\frac{n-2}{n}}+\frac{qaC(n, p,B)}{2}\int(\eta w)^2+\frac{C(\tilde{e})aq}{2}\int(\eta w)^2.
\end{equation}
Applying (\ref{1.4}) to $f=\eta w$ for $r=1$, then from (\ref{4.2}) and (\ref{4.6}) we have
\begin{equation}\label{4.7}
\begin{split}
\frac{1}{C_S^2}\left( \fint(\eta w)^\frac{2n}{n-2}\right)^{\frac{n-2}{n}}\leq&\frac{qaB}{2}\left(\fint (\eta w)^{\frac{2n}{n-2}}\right)^{\frac{n-2}{n}}+\frac{qaC(n, p,B)}{2}\fint(\eta w)^2+\frac{C(\tilde{e})aq}{2}\fint(\eta w)^2\\
&+\fint\left| \nabla\eta\right|^2w^2+q\fint(b)_+(\eta w)^2.
\end{split}
\end{equation}
Let $qaB=\frac{1}{C_S^2}$, then (\ref{4.7}) changes into
\begin{equation}\label{4.8}
\begin{split}
\left( \fint(\eta w)^\frac{2n}{n-2}\right)^{\frac{n-2}{n}}\leq&C_S^2\left( qaC(n, p,B)\fint(\eta w)^2+C(\tilde{e})aq\fint(\eta w)^2\right) \\
&+2C_S^2\left( \fint\left| \nabla\eta\right|^2w^2+q\fint(b)_+(\eta w)^2\right).
\end{split}
\end{equation}

For $n=2$, during the proof of Theorem \ref{th1.2}, we have pointed out that it can be viewed as the same as the case of $n=3$ by letting $p>\frac{3}{2}$, i.e. if we define
\begin{equation*}
n'=
\left\{
\begin{array}{lr}
n ,&n\geq3,\\
3 ,&n=2,
\end{array}\quad and \quad
\right.	
\mu=
\left\{
\begin{array}{lr}
\frac{n}{n-2} ,&n\geq3,\\
3 ,&n=2,
\end{array}
\right.
\end{equation*}
then from (\ref{4.8}), for $p>\frac{n'}{2}$ we have
\begin{equation}\label{4.9}
 \begin{split}
 \left( \fint(\eta w)^{2\mu}\right)^\mu\leq&C_S^2\left( qaC(n, p,B)\fint(\eta w)^2+C(\tilde{e})aq\fint(\eta w)^2\right) \\
 &+2C_S^2\left( \fint\left| \nabla\eta\right|^2w^2+q\fint(b)_+(\eta w)^2\right).
 \end{split}
\end{equation}
By a similar argument with (\ref{2.20}), there exists a constant $C_1=C_1\left( n',p,C_S^2,q^{\frac{2p}{2p-n'}},\left| \left|(b)_+ \right| \right|^*_{p, B_1}\right)$ such that
\begin{equation}\label{4.10}
2C_S^2q\fint(b)_+(\eta w)^2\leq\frac{1}{2}\left( \fint(\eta w)^{2\mu}\right)^\mu+C_1\fint(\eta w)^2.
\end{equation}
Plugging (\ref{4.10}) into (\ref{4.9}), there holds
\begin{equation}\label{4.11}
 \left( \fint(\eta w)^{2\mu}\right)^\mu\leq2C_S^2\left( qaC(n, p,B)\fint(\eta w)^2+\left( C(\tilde{e})aq+C_1\right) \fint(\eta w)^2\right)+4C_S^2 \fint\left| \nabla\eta\right|^2w^2.
\end{equation}
Then, for some constant $C_2=C_2\left( n,p,a,C_S,\left| \left|(b)_+\right| \right|^*_{p, B_1}\right) >0$ we can obtain
$$\left| \left| u\right| \right|_{\infty,B_{\frac{1}{2}}}\leq C_2\left| \left| u\right| \right|^*_{1,B_\frac{4}{5}}$$
by iterating just like the procedures after (\ref{2.26}). Hence we obtain (\ref{1.8**}) by the fact that \begin{equation}\label{4.12**}
\left\| u\right\|^*_{1,B_{\frac{4}{5}}}\leq\left\| u\right\|^*_{2,B_{\frac{4}{5}}}=\left( \frac{\int_{B_\frac{4}{5}}u^2}{B_{\frac{4}{5}}}\right)^\frac{1}{2}\leq\frac{\left\| u\right\|_{2,B_1}}{V^{\frac{1}{2}}}.
\end{equation}
Thus, we finish the proof of the first claim of Theorem \ref{th1.2}.
\medskip

\noindent(2). Since by (\ref{1.8**}) we have $u\leq D$ on  $B\left( y,\frac{1}{2}\right) \subset B_1$ for any $y\in B_1$, so we have $u\leq D$ on the whole $B_1$. Let $R_c=\left| Ric^-\right|$ and  $v=\left| \nabla u\right|^2$. As before, we may assume $v>1$. Since $a>0$, by Bochner formula we have
\begin{equation*}
\begin{split}
	\frac{1}{2}\Delta\left| \nabla u\right|^2&=\left| D^2u\right|^2+\left\langle \nabla u, \nabla\Delta u\right\rangle+Ric\left( \nabla u, \nabla u\right)\\
	&\geq\left\langle \nabla u, \nabla\Delta u\right\rangle+Ric\left( \nabla u, \nabla u\right)\\
	&\geq-\left\langle \nabla u, \nabla(au\log u+bu)\right\rangle-R_cv\\
	&=-av\log u-av-bv-u\left\langle \nabla u, \nabla b\right\rangle-R_cv\\
	&\geq -av\log D-av-bv-\frac{u}{2}\left( v+\left| \nabla b\right|^2\right)-R_cv\\
	&\geq -av\log D-av-bv-\frac{D}{2}\left( v+\left| \nabla b\right|^2\right)-R_cv,
\end{split}
\end{equation*}
and since $v>1$, there holds
\begin{equation}\label{4.12}
\frac{1}{2}\Delta v\geq-a\left( 1+log D+\frac{D}{2}\right) v-\frac{D\left| \nabla b\right|^2v}{2}-(b)_+v-R_cv.
\end{equation}
Next, for any $l\geq0$ and $\eta\in C^\infty_0(B_1)$, we multiply by $\eta^2v^l$ on both sides of (\ref{4.12}) to get
\begin{equation}\label{4.13}
\int\frac{1}{2}\eta^2v^l\Delta v\geq\int\left( -a(1+log D+\frac{D}{2})\eta^2v^{l+1}-\frac{D\left| \nabla b\right|^2\eta^2v^{l+1}}{2}-(b)_+v^{l+1}-R_c\eta^2v^{l+1}\right) .
\end{equation}
Now, (\ref{4.13}) is similar to (\ref{2.13}) but simpler, the only difference is that we need to handle the following two terms
$$\int\left| \nabla b\right|^2\eta^2v^{l+1}\quad \mbox{and} \quad \int (b)_+\eta^2v^{l+1}$$
as in (\ref{2.20}). Then, following the same iteration as in the proof of Theorem \ref{th1.2}, we obtain that there exists a constant $C=C\left( n,p,a,D,\kappa(p,1),C_S,\left| \left|(b)_+\right| \right|^*_{p, B_1},\left| \left|\left| \nabla b\right|^2  \right| \right|^*_{p, B_{1}}\right)$ such that
\begin{equation}\label{4.14}
\sup\limits_{B_\frac{1}{2}}v\leq C\left\| v\right\|^*_{1,B_{\frac{4}{5}}}.
\end{equation}
Then by using the same inequality to $v$ as in (\ref{4.12**}), we obtain the required (\ref{1.9**}) since $u\in W^{1,2}(B_1)$.
\end{proof}

\section{\textbf{Proof of Theorem \ref{th1.3'}}}\label{sec5}
In order to overcome the difficulty that $\log u$ maybe unbounded, we consider $u^{\frac{1}{q}}$ with $q>1$ instead of $\log u$.
\begin{proof}
Set $w=u^{\frac{1}{q}}$, from (\ref{equ}) we know $w$ satisfies
\begin{equation}\label{5.1}
\Delta w+(q-1)\frac{\left| \nabla w\right|^2}{w}+aw\log w+\frac{bw}{q}=0.
\end{equation}
Let $v=(q-1)\frac{\left| \nabla w\right|^2}{w}$ and $b'=\frac{b}{q}$. As before, we can also assume $v\geq1$ since otherwise it's suitable for us to consider $v+1$ instead of $v$.
By Bochner formula we have
\begin{align}
\Delta v=&(q-1)\Delta\left( \frac{\left| \nabla w\right|^2}{w}\right)\nonumber\\
=&(q-1)\left[ \frac{\Delta\left| \nabla w\right|^2}{w}+\left| \nabla w\right|^2\Delta\left( \frac{1}{w}\right)+2\left\langle \nabla(\left| \nabla w\right|^2), \nabla\left(\frac{1}{w} \right) \right\rangle  \right] \nonumber\\
\geq&(q-1)\left[ \frac{2\left| D^2 w\right|^2+2\left\langle\nabla w, \nabla\Delta w \right\rangle - 2Ric^{-}\left| \nabla w\right|^2 }{w}+\left| \nabla w\right|^2\left( -\frac{\Delta w}{w^2}+\frac{2\left| \nabla w\right|^2}{w^3}\right)\right]\label{5.2}\\
& +2(q-1)\left\langle \nabla\left| \nabla w\right|^2, -\frac{\nabla w}{w^2} \right\rangle.\nonumber
\end{align}
Substituting $\Delta w=-v-aw\log w-b'w$ and $\left| \nabla w\right|^2=\frac{vw}{q-1}$ into (\ref{5.2}) gives
\begin{equation}\label{5.3}
\begin{split}
\Delta v=&(q-1)\left[ \frac{2\left| D^2 w\right|^2}{w}+\frac{v^2}{(q-1)w}-\left( \frac{2}{w}+\frac{2}{(q-1)w}\right)\left\langle\nabla w,\nabla v \right\rangle -2\left\langle \nabla w, \nabla b'\right\rangle \right]\\
&+\left(-2Ric^{-}-a\log w-b'-2a\right) v.
\end{split}
\end{equation}
Since $-2\left\langle \nabla w, \nabla b'\right\rangle\geq-\frac{\left| \nabla w\right|^2}{w}-w\left| \nabla b'\right|^2$, (\ref{5.3}) leads to
\begin{equation}\label{5.4}
\begin{split}
\Delta v\geq&(q-1)\left[ \frac{v^2}{(q-1)w}-\left( \frac{2}{w}+\frac{2}{(q-1)w}\right)\left\langle\nabla w,\nabla v \right\rangle\right]\\
&+\left(-2Ric^{-}-a\log D^{\frac{1}{q}}-b'-2a-1\right)v-(q-1)D^{\frac{1}{q}}\left| \nabla b'\right|^2.
\end{split}
\end{equation}
Next, for any $l\geq0$ and $\eta\in C^\infty_0\left( B_\lambda\right) $, since $v\geq1$, we multiply by $ \frac{1}{2}\eta^2v^l $ on both sides of (\ref{5.4}) and integrate on $B(x, \lambda)$ to get
\begin{equation}\label{5.5}
\begin{split}
\frac{1}{2}\int\eta^2v^l\Delta v\geq&\int\left( \frac{\eta^2v^{l+2}}{2w}-q\eta^2v^l\left\langle\nabla \log w, \nabla v \right\rangle-R_c\eta^2v^{l+1}-C_1\eta^2v^{l+1}\right)\\
&-\int\left( C_2\left| \nabla b'\right|^2\eta^2v^{l+1}
+\frac{b'}{2}\eta^2v^{l+1}\right)
\end{split}
\end{equation}
where $$C_1=\frac{a\left( \log D^{\frac{1}{q}}\right)_++2a+1}{2},\quad\quad C_2=\frac{(q-1)D^{\frac{1}{q}}}{2}$$ and $R_c=\left|Ric^- \right|$. From now on, if there has no special emphasis on integral domain, we always calculate on $B_\lambda$.

By Green formula and Cauchy-Schwartz inequality, there holds
\begin{align}
\int\eta^2v^l\left\langle \nabla v, \nabla \log w\right\rangle&=-\frac{1}{l+1}\int\eta^2v^{l+1}\Delta \log w-\frac{2}{l+1}\int v^{l+1}\left\langle \nabla\eta, \eta\nabla \log w\right\rangle\nonumber\\
&\leq -\frac{1}{l+1}\int\eta^2v^{l+1}\left( \frac{\Delta w}{w}-\frac{\left| \nabla w\right|^2}{w^2}\right)+\frac{1}{l+1}\int v^{l+1}\left(\left|\nabla\eta \right|^2 +\eta^2\frac{\left| \nabla w\right|^2}{w^2}\right)\label{5.6}.
\end{align}
Substituting $\Delta w=-v-aw\log w-b'w$ and $\left| \nabla w\right|^2=\frac{vw}{q-1}$ into (\ref{5.6}) gives
\begin{equation}\label{5.7}
\begin{split}
&\int\eta^2v^l\left\langle \nabla v, \nabla \log w\right\rangle\\
&\leq\left( \frac{1}{l+1}+\frac{2}{(l+1)(q-1)}\right)\int\frac{\eta^2v^{l+2}}{w}+\frac{1}{l+1}\left( \int C_3\eta^2v^{l+1}+\int\left| \nabla\eta \right|^2v^{l+1}\right)
\end{split}
\end{equation}
where $$C_3=\left( a\log D^{\frac{1}{q}}\right)_++(b')_+.$$
Combining (\ref{5.5}) with (\ref{5.7}), we have
\begin{equation}\label{5.8}
\begin{split}
&\int\frac{1}{2}\eta^2v^l\Delta v\\
&\geq\int\left[ \frac{1}{2}-\left(\frac{q}{l+1}+\frac{2q}{(l+1)(q-1)} \right)\right]\frac{\eta^2v^{l+2}}{w} -\frac{q}{l+1}\left( \int C_3\eta^2v^{l+1}+\int\left| \nabla\eta \right|^2v^{l+1}\right) \\
&\quad-\int R_c\eta^2v^{l+1}-\int C_1\eta^2v^{l+1}-\int\left( C_2\left| \nabla b'\right|^2\eta^2v^{l+1}
+\frac{b'}{2}\eta^2v^{l+1}\right).
\end{split}
\end{equation}
Let $l+1\geq2q+\frac{4q}{q-1}$, then $\frac{1}{2}-\left(\frac{q}{l+1}+\frac{2q}{(l+1)(q-1)}\right)\geq0$, so (\ref{5.8}) becomes
\begin{equation}\label{5.9}
\begin{split}
\int\frac{1}{2}\eta^2v^l\Delta v
\geq&-\frac{q}{l+1}\left( \int C_3\eta^2v^{l+1}+\int\left| \nabla\eta \right|^2v^{l+1}\right)-\int R_c\eta^2v^{l+1}-\int C_1\eta^2v^{l+1}\\
&-\int\left( C_2\left| \nabla b'\right|^2\eta^2v^{l+1}
+\frac{b'}{2}\eta^2v^{l+1}\right).
\end{split}
\end{equation}
By (\ref{2.14}) and (\ref{5.9}), there holds
\begin{equation}\label{5.10}
\begin{split}
&\int\left| \nabla\left( \eta v^{\frac{l+1}{2}}\right) \right|^2\\
&\leq\frac{(l+1)^2}{l}\left[\frac{q}{l+1}\left( \int C_3\eta^2v^{l+1}+\int\left| \nabla\eta \right|^2v^{l+1}\right)+\int R_c\eta^2v^{l+1}+\int C_1\eta^2v^{l+1} \right]\\
&\quad+\frac{(l+1)^2}{l}\int\left( C_2\left| \nabla b'\right|^2\eta^2v^{l+1}
+\frac{b'}{2}\eta^2v^{l+1}\right)\\
&\quad+\frac{\left( l+1\right)^2+l }{l^2} \int v^{l+1}\left| \nabla\eta\right|^2-\frac{l+1}{l}\int \eta v^{l+1}\Delta \eta.
\end{split}
\end{equation}
For $0<r<1$, let $\phi\in C^\infty_0\left( [ 0, +\infty)\right) $ be a cut-off function such that
$$0\leq\phi\leq1, \quad\hspace*{0.3em}\phi(t)\equiv1\hspace*{0.3em} \mbox{for}\hspace*{0.3em}0\leq t\leq r\lambda,\quad\hspace*{0.3em}\phi(t)\equiv0\hspace*{0.3em}\mbox{for}\hspace*{0.3em}t\geq\lambda\hspace*{0.3em}\mbox{and}\hspace*{0.3em}\phi'\leq0.$$
Let $\eta(y)=\phi(d(x,y))$, then $\left| \nabla\eta\right|=\left| \phi'\right|$ and for $\psi(y)=\left(\Delta d-\frac{n-1}{d} \right)_+ $ there holds
\begin{equation}\label{5.11}
\begin{split}
\Delta\eta&=\phi''+\phi'\Delta d=\phi''+\phi'\left( \Delta d-\frac{n-1}{d}+\frac{n-1}{d}\right)\\
&\geq\phi''+\phi'\left( \psi+\frac{n-1}{d}\right)\geq-\left|\phi'' \right|-\frac{(n-1)\left|\phi' \right| }{r\lambda}-\left| \phi'\right|\psi.
\end{split}
\end{equation}
 As a consequence, we deduce from (\ref{5.10}) and (\ref{5.11}) that there exists a constant $C_4=C_4(q,a,D)$ such that
 \begin{equation}\label{5.12}
 \begin{split}
&\int\left| \nabla\left( \eta v^{\frac{l+1}{2}}\right) \right|^2\\
&\leq C_4l\int\left[ \left( \left| \phi''\right| +\frac{\left|\phi' \right| }{r\lambda}+\left|  \phi'\right|\psi \right) \eta v^{l+1}+\left| \phi'\right|^2v^{l+1}+R_c\eta^2v^{l+1}+\eta^2v^{l+1}\right]\\
&\quad+C_4l\int\left( \left| \nabla b'\right|^2\eta^2v^{l+1}
+b'\eta^2v^{l+1}\right).
 \end{split}
 \end{equation}
Using (\ref{1.3}) and (\ref{5.12}), we obtain
\begin{equation}\label{5.13}
\begin{split}
&\left( \fint\left( \eta^2v^{l+1}\right)^\mu \right)^{\frac{1}{\mu}}\\
&\leq C^2(n)\lambda^2C_4l\fint\left[ \left( \left| \phi''\right| +\frac{\left|\phi' \right| }{r\lambda}+\left|  \phi'\right|\psi \right) \eta v^{l+1}+\left| \phi'\right|^2v^{l+1}+R_c\eta^2v^{l+1}+\eta^2v^{l+1}\right]\\
&\quad+C^2(n)\lambda^2C_4l\fint\left( \left| \nabla b'\right|^2\eta^2v^{l+1}
+b'\eta^2v^{l+1}\right)
\end{split}
\end{equation}
where $C_S(n,\lambda)=C(n)\lambda$.
Following the same program as we got (\ref{2.21}) and (\ref{2.25}) in the previous, we can see that there holds
\begin{equation}\label{5.14}
\begin{split}
&\left(\fint\left( \eta^2v^{l+1}\right)^\mu  \right)^{\frac{1}{\mu}}\\
&\leq6C^2(n)\lambda^2C_4l\left[ \fint\left( \left| \phi''\right| +\frac{\left|\phi' \right| }{r\lambda} \right) \eta v^{l+1}+\left( 1+\frac{3}{2}C^2(n)\lambda^2C_4lC(n, p)\left| \left| R_c\right| \right|^*_p\right)\fint \left| \phi'\right|^2v^{l+1}\right]\\
&\hspace*{1em}+6C_5\left( C^2(n)\lambda^2l\right)^{\frac{2p}{2p-n'}}\left( 1+\left(\left| \left| R_c\right| \right|^*_p  \right)^\frac{2p}{2p-n'}+\left(\left| \left| b'\right| \right|^*_p  \right)^\frac{2p}{2p-n'}+\left(\left| \left| \left| \nabla b'\right|^2\right|  \right|^*_p  \right)^\frac{2p}{2p-n'}\right)\fint\eta^2v^{l+1}
\end{split}
\end{equation}
for some constant $C_5=C_5(C_4,n,p)$. Here we still use the symbols
\begin{equation*}
n'=
\left\{
\begin{array}{lr}
n ,&n\geq3\\
3 ,&n=2
\end{array}\quad and \quad
\right.	
\mu=
\left\{
\begin{array}{lr}
\frac{n}{n-2} ,&n\geq3\\
3 ,&n=2
\end{array}
\right.
\end{equation*}
and estimate the quantities $\fint\left| \nabla b'\right|^2\eta^2v^{l+1}$ and $\fint b'\eta^2v^{l+1}$ by the same route as we dealt with the term $\fint R_c\eta^2v^{l+1}$.

For integers $j\geq0$, we define $$r_j=\lambda\left( \theta+\delta-\delta\left( \sum^j\limits_{i=0}2^{-i-1}\right) \right)$$
where $\theta,\delta$ are two constants such that $0<\delta\leq\frac{1}{2}$ and $\frac{1}{2}\leq\theta\leq\frac{4}{5}-\delta$, then $\theta\lambda<r_j<(\theta+\delta)\lambda$. Moreover, we can choose $\eta_j(y)=\phi_j(d(y))\in C^\infty_0(B_{r_{j-1}})$ such that
$$\phi_j\equiv1\quad\mbox{on}\hspace*{0.3em}B_{r_j},\quad \hspace*{0.3em}\phi_j\equiv0\quad\mbox{on}\hspace*{0.3em}B_1\setminus B_{r_{j-1}},\quad B_{r_{-1}}=B_{(\theta+\delta)\lambda}$$
and
$$\left|\phi_j'\right|\leq\frac{2^{j+3}}{\delta\lambda},\quad\quad \left| \phi_j''\right|\leq\frac{2^{2j+6}}{\delta^2\lambda^2}.$$
In this situation, we have $$\frac{\left|\phi_j'\right|}{r_j\lambda}\leq\frac{2^{2j+6}}{\delta^2\lambda^2},\quad\quad \left|\phi_j'\right|^2\leq\frac{2^{2j+6}}{\delta^2\lambda^2},\quad\quad\eta_{j+1}\leq\eta^\mu_j\leq\eta_j$$ and $\eta^\mu_j\equiv1$ on the  support of $\eta_{j+1}$.

Next, it is easy that there always exists a minimal integer $m$ such that $$\mu^m\geq2q+\frac{4q}{q-1}$$
and let $q'=\max{\left\lbrace 2,\frac{2p}{2p-n'}\right\rbrace }$. For simplicity, we denote \begin{equation*}
\begin{split}
B\equiv&\delta^{-2}\left[ 6C_5\left( C^2(n)\right)^{\frac{2p}{2p-n'}}\left( 1+\left(2^{\frac{1}{p}}\kappa(p,1)  \right)^\frac{2p}{2p-n'}+\left(\left| \left| b'\right| \right|^*_p  \right)^\frac{2p}{2p-n'}+\left(\left| \left| \left| \nabla b'\right|^2\right|  \right|^*_p  \right)^\frac{2p}{2p-n'}\right)\right] \\
&+\delta^{-2}\left[ 6C^2(n)C_4\left(3+\frac{3}{2}C^2(n)C_4C(n, p)2^{\frac{1}{p}}\kappa(p,1)\right)\right]
\end{split}
\end{equation*}
and
\begin{equation*}
	\begin{split}
		A\equiv&6C_5\left( C^2(n)\right)^{\frac{2p}{2p-n'}}\left( 1+\left(2^{\frac{1}{p}}\kappa(p,1)  \right)^\frac{2p}{2p-n'}+\left(\left| \left| b'\right| \right|^*_p  \right)^\frac{2p}{2p-n'}+\left(\left| \left| \left| \nabla b'\right|^2\right|  \right|^*_p  \right)^\frac{2p}{2p-n'}\right)\\
		&+ 6C^2(n)C_4\left(3+\frac{3}{2}C^2(n)C_4C(n, p)2^{\frac{1}{p}}\kappa(p,1)\right).
	\end{split}
\end{equation*}
Then, according to (\ref{5.14}), Remark \ref{r1.2} and noting $\lambda^\alpha\leq\lambda^2\leq1$ for $\alpha\geq2$, there holds true that
\begin{equation}\label{5.15}
\begin{split}
\left(\fint\left( \eta^2_{j+1}v^{l+1}\right)^\mu  \right)^{\frac{1}{\mu}}\leq Bl^{q'}2^{2j+6}\fint \eta^{2\mu}_jv^{l+1}\leq B(l+1)^{q'}2^{2j+6}\fint \eta^{2\mu}_jv^{l+1}.
\end{split}
\end{equation}
Then following the same iteration steps as to get (\ref{2.32}), letting $t=\mu^m$, $s=1$, $R=\lambda$ and $\tau=\frac{3}{10}$, by Lemma \ref{l2.3} one can see that there exists a constant $C_6=C_6(\mu,n,q,q',A)$ which does not depend on $\mu^m$ such that
\begin{equation}\label{5.16}
\left| \left| v\right| \right|_{\infty,B_{\frac{\lambda}{2}}}\leq C_6\left| \left| v\right| \right|^*_{1,B_\frac{4\lambda}{5}}.
\end{equation}

Next, we choose some $\eta\in C^\infty_0(B_\lambda)$ such that
$$\eta\equiv1\quad\mbox{on}\hspace*{0.3em}B_\frac{4\lambda}{5}\quad\mbox{and}\quad \left| \nabla\eta\right|\leq\frac{10}{\lambda}.$$
For $\eta^2v$, there holds
\begin{equation}\label{5.17}
\begin{split}
\int\eta^2v=-\int\eta^2\left( \Delta w+aw\log w+b'w\right).
\end{split}
\end{equation}
By Green's formula, we have
$$-\int\eta^2\Delta w=2\int\left\langle \nabla\eta, \eta\nabla w\right\rangle\leq\frac{1}{2}\int\eta^2v+2\int\frac{w\left| \nabla\eta\right|^2}{q-1}\leq\frac{1}{2}\int\eta^2v+\frac{200D^{\frac{1}{q}}}{(q-1)\lambda^2}\left| B_\lambda\right|.$$
On the other hand, we also have
$$-\int\eta^2(aw\log w+b'w)\leq\left( \frac{a}{e}+\left| \left|b'\right| \right|^*_{p, B_\lambda}D^{\frac{1}{q}}\right)\left|B_\lambda \right|.$$
Here, we have used the following facts that for $w\in(0, +\infty)$ there hold
$$-w\log w\leq\frac{1}{e}$$ and the monotonicity inequality
$$\left| \left|b'\right| \right|^*_{1, B_\lambda}\leq\left| \left| b'\right| \right|^*_{p, B_\lambda}.$$

Now, from (\ref{5.17}) we can deduce that there holds
\begin{equation}\label{5.18}
\begin{split}
\left| \left| v\right| \right|^*_{1,B_\frac{4\lambda}{5}}\leq\frac{\int\eta^2v}{\left|B_{\frac{4\lambda}{5}}\right|}&\leq\frac{\left( \frac{400D^{\frac{1}{q}}}{(q-1)\lambda^2}+\frac{2a}{e}+2\left| \left|b'\right| \right|^*_{p, B_\lambda}D^{\frac{1}{q}}\right)\left|B_\lambda \right|}{\left|B_{\frac{4\lambda}{5}} \right|}\\
&\leq\frac{\left( \frac{400D^{\frac{1}{q}}}{(q-1)}+\frac{2a}{e}+2\left| \left|b'\right| \right|^*_{p, B_\lambda}D^{\frac{1}{q}}\right)\left|B_\lambda \right|}{\lambda^2\left|B_{\frac{4\lambda}{5}} \right|}.
\end{split}
\end{equation}
By (\ref{2.2}), we have
\begin{equation}\label{5.19}
\left|B_\lambda \right|\leq\frac{2\cdot5^n}{4^n}\left|B_{\frac{4\lambda}{5}}\right|.
\end{equation}
Substituting (\ref{5.18}) into (\ref{5.16}) and noting $$\frac{\left|\nabla w \right|^2}{w}=\frac{\left|\nabla u \right|^2u^{\frac{1}{q}-2}}{q^2},$$ we get the required estimate
\begin{equation}\label{5.20}
\sup\limits_{B\left( x,\frac{\lambda}{2}\right) }\frac{\left|\nabla u\right|}{u^{1-\frac{1}{2q}}} \leq \frac{C}{\lambda}.
\end{equation}
Thus, we accomplish the proof of Theorem \ref{th1.3'} provided $k$ is small enough such that Theorem \ref{th1.1} and Lemma \ref{l2.2} are valid.
\end{proof}

\section{\textbf{Global Estimates}}
It's well-known that for complete non-compact Riemannian manifold $(M, g)$ with Ricci curvature $Ric_M\geq0$ (i.e. $R_c\equiv0$), by Theorem 14.3 in \cite{Li}, the following Sobolev inequality holds true for any $B(x,R)\in M$:
\begin{equation}\label{6.1}
	\left| \left| f\right| \right|^2_{\frac{2n}{n-2},B_R}\leq C^2_SR^2\left| B_R\right|^{-\frac{2}{n}} \left| \left| \nabla f\right| \right|^2_{2,B_R}
\end{equation}
for any $f\in C^{\infty}_0(B_R)$, and consequently, for $f\in W_0^{1,2}(B_R)$. Especially, the Sobolev constant $C_S=C_S(n)$ does not depend on the radius $R$. However, in order to achieve the expected goal, we need to transform the above Sobolev inequality to the type similar to (\ref{1.4}).

As usual, we define $\mu=\frac{n}{n-2}$. Dividing by $\left| B_R\right| ^{\frac{1}{\mu}}$ on both sides of (\ref{6.1}), we get
\begin{equation}\label{6.2}
	\left| \left| f\right| \right|^*_{\frac{2n}{n-2},B_R}\leq C_SR \left| \left| \nabla f\right| \right|^*_{2,B_R}.
\end{equation}
Now, since (\ref{6.2}) is almost the same as (\ref{1.4}), so the local estimate in Theorem \ref{th1.3'} can be extended to the whole $M$ easily. In the following, we use the same symbols as in section \ref{sec5} if there is no special emphasis.
\begin{proof}[Proof of Theorem \ref{th1.4'}]
Letting $w=u^{\frac{1}{q}}$ for $q>1$ and calculating on $B_R$ as in section \ref{sec5}. Since $\left| \nabla b\right|\equiv0$, we need not to assume $v\geq1$ to obtain (\ref{5.5}) from (\ref{5.4}), hence there holds
\begin{equation}\label{6.8}
\begin{split}
&\int\left| \nabla\left( \eta v^{\frac{l+1}{2}}\right) \right|^2\\
&\leq\frac{(l+1)^2}{l}\left[\frac{q}{l+1}\left( \int C_3\eta^2v^{l+1}+\int\left| \nabla\eta \right|^2v^{l+1}\right)+\int C_1\eta^2v^{l+1}+\int\frac{b'}{2}\eta^2v^{l+1} \right]\\
&\quad+\frac{\left( l+1\right)^2+l }{l^2} \int v^{l+1}\left| \nabla\eta\right|^2-\frac{l+1}{l}\int \eta v^{l+1}\Delta \eta
\end{split}
\end{equation}
if $l+1\geq2q+\frac{4q}{q-1}$ at least. Here $C_1=\frac{a\log D^{\frac{1}{q}}+2a}{2}$, $C_2=0$ and $C_3=a\log D^{\frac{1}{q}}+b'$ since $b$ is constant, and to make our estimate more precise, we do not take their positive parts.
	
	Next, we need to use the cut-off function introduced by Li-Yau in \cite{LY}. Concretely, let $\phi(r)$ be a nonnegative $C^2$-smooth function on $R^+=\left[0,+\infty \right)$ such that $\phi(r)=1$ for $r\leq \frac{1}{2}$ and $\phi(r)=0$ for $ r\geq1$. Moreover, there exist two positive constants $C_4$ and $C_5$ such that the derivatives of $\phi(r)$ satisfy the conditions as follows:
	\begin{equation}\label{6.9}
		-C_4\phi^{\frac{1}{2}}(r)\leq\phi^{'}(r)\leq 0\quad\quad \mbox{and}\quad\quad -C_5\leq\phi^{''}(r).
	\end{equation}

Now, let $\eta(y)=\phi\left(\frac{d(y,x)}{R}\right)$ where $d(y,x)$ denotes the distance from $y$ to $x$ and it is obvious that $\eta(y)$ is supported in $B_R$:
	\begin{align*}
		&\eta|_{B_\frac{R}{2}}=1,\\
		&\eta|_{M\backslash B_R}=0.
	\end{align*}
	Consequently, it follows from (\ref{6.9}) and the Laplacian comparison theorem that
	\begin{align}
		&\frac{\left| \nabla  \eta\right|^2 }{\eta}\leq\dfrac{C_4^2}{R^2},\label{6.10}\\
		&\Delta\eta\geq-\frac{(n-1)C_4^2+C_5}{R^2}.\label{6.11}
	\end{align}
	
Then, by (\ref{6.8}), (\ref{6.10}) and (\ref{6.11}) we know that there holds
	\begin{align}\label{6.12}
		&\int\left| \nabla\left( \eta v^{\frac{l+1}{2}}\right) \right|^2\nonumber\\
		&\leq\frac{(l+1)q}{l}\int\left| \nabla\eta \right|^2v^{l+1}+\frac{\left( l+1\right)^2+l }{l^2} \int v^{l+1}\left| \nabla\eta\right|^2-\frac{l+1}{l}\int \eta v^{l+1}\Delta \eta \nonumber\\
		&\quad+\frac{(l+1)^2}{l}\left[\left( \frac{qC_3}{l+1}+C_1+\frac{b'}{2}\right) \int \eta^2v^{l+1}\right]\nonumber\\
		&\leq\left( \frac{(l+1)q}{l}+\frac{\left( l+1\right)^2+l }{l^2}\right)\int \frac{C_4^2\eta v^{l+1}}{R^2}+\frac{l+1}{l}\int \eta v^{l+1}\left( \frac{(n-1)C_4^2+C_5}{R^2}\right)\nonumber\\
		&\quad+\frac{(l+1)^2}{l}\left[\left( \frac{qC_3}{l+1}+C_1+\frac{b'}{2}\right) \int \eta^2v^{l+1}\right].	
	\end{align}
As a consequence of Sobolev inequality (\ref{6.2}), letting $f=\eta v^{\frac{l+1}{2}}$ in (\ref{6.12}) reveals that
\begin{align}\label{6.13}
&\left(\fint\left( \eta^2v^{l+1}\right)^\mu  \right)^{\frac{1}{\mu}}\nonumber\\
	&\leq C_S^2R^2\left[ \left( \frac{(l+1)q}{l}+\frac{\left( l+1\right)^2+l }{l^2}\right)\int \frac{C_4^2\eta v^{l+1}}{R^2}+\frac{l+1}{l}\int \eta v^{l+1}\frac{(n-1)C_4^2+C_5}{R^2}\right] \nonumber\\
	&\quad+\frac{(l+1)^2C_S^2R^2}{l}\left[\left( \frac{qC_3}{l+1}+C_1+\frac{b'}{2}\right) \int \eta^2v^{l+1}\right]\nonumber\\
	&\leq C_S^2R^2\left[ \left( \frac{(l+1)q}{l}+\frac{\left( l+1\right)^2+l }{l^2}\right)\int \frac{C_4^2\eta v^{l+1}}{R^2}+\frac{l+1}{l}\int \eta v^{l+1}\frac{(n-1)C_4^2+C_5}{R^2}\right]\nonumber\\
	&\quad+\frac{(l+1)^2C_S^2R^2}{l}\left(  \frac{qC_3}{l+1} \int \eta^2v^{l+1}\right) 	
\end{align}
since $a\log D+2aq+b\leq0$ ensures $C_1+\frac{b'}{2}\leq0$.
	
Now, it's easy to see from (\ref{6.13}) that there exists a constant $C_6$ such that
\begin{equation}\label{6.14}
\left(\fint\left( \eta^2v^{l+1}\right)^\mu  \right)^{\frac{1}{\mu}}\leq C_6C_S^2(l+1)\left[\eta v^{l+1}+\frac{R^2\eta^2v^{l+1}}{l+1}\right],
\end{equation}	
where $C_6=C_6(n,q,a,b,D,C_4,C_5,C_S)$. Let $(l+1)\geq max\left\lbrace 2q+\frac{4q}{q-1}, R^2\right\rbrace$, then (\ref{6.14}) gives
\begin{equation}
\left(\fint\left( \eta^2v^{l+1}\right)^\mu  \right)^{\frac{1}{\mu}}\leq C_6C_S^2(l+1)\left[\eta v^{l+1}+\eta^2v^{l+1}\right].
\end{equation}
By Bishop-Gromov's volume comparison theorem, we also have the volume doubling property $$\left|B_R \right|\leq\frac{5^n}{4^n}\left|B_{\frac{4R}{5}}\right|.$$	
Then just following the same iteration procedures as to obtain (\ref{5.20}), we also get that there exists a constant $C_7=C_7(n,q,a,b,D,C_4,C_5,C_S)$ which does not depend on the radius $R$ such that
\begin{equation}\label{6.16}
\sup\limits_{B\left( x,\frac{R}{2}\right) }\frac{\left|\nabla u\right|}{u^{1-\frac{1}{2q}}}\leq\frac{C_7}{R}.
\end{equation}
	
Letting $R\rightarrow+\infty$, we derive that $\left|\nabla u\right|\equiv0$ on the whole $M$ since $u$ is bounded, in other words, $u$ must be constant.	
\end{proof}

\section{\textbf{Proof of Theorem \ref{th1.3}}}
In this section, we denote $B(x,r)$ by $B_r$, $\lvert Ric^-\rvert$ by $R_c$. To give our proof, we need a lemma from the proof of Theorem 1.1 in \cite{ZZ}.
\begin{lemma}\label{l3.1}
	For positive smooth functions $J(y, t)$ on $M\times[0,\infty)$, there exists a constant $k=k(n,p)$ such that if $\kappa(p, 1)\leq k$, then for any $0<r\leq1$, the equation
	\begin{equation*}
		\left\{
		\begin{array}{lr}
			\Delta J-2R_cJ-5\delta^{-1}\frac{\lvert \nabla J\rvert^2 }{J}-\partial_tJ=0\quad\hspace*{0.3em}\mbox{on}\hspace*{0.3em}B_r\times(0, \infty),\\
			J(\cdot, 0)=1\quad\hspace*{0.3em}\mbox{on}\hspace*{0.3em}B_r,\\
			J(\cdot, t)=1\quad\hspace*{0.3em}\mbox{on}\hspace*{0.3em}\partial B_r.
		\end{array}
		\right.
	\end{equation*}
	has a unique solution for $t\in[0, \infty)$ which satisfies
	\begin{equation}\label{3.1}
		0<\underline{J_r}(t)\leq J(y,t)\leq1
	\end{equation}
	where
	\begin{equation*}
		\underline{J_r}=\underline{J_r}(t)=2^{-\frac{1}{C_1-1}}\exp\left\lbrace -2C_2k\left( 1+\left[ 2C_2(C_1-1)k\right]^{\frac{n}{2p-n}} \right)t \right\rbrace.
	\end{equation*}
	Here $C_1$ and $C_2$ are the same as in Theorem \ref{th1.3}.
\end{lemma}

The basic idea in our proof is the maximum principle, and to this end, we need to establish the following lemma.
\begin{lemma}\label{l3.2}
	Let $0<J\leq1$ be some function on $M\times[0, \infty)$ and define $$F=J\lvert \nabla \rvert^2+(A+a)f+2(N+b)-2f_t$$ for some constants $A$ and $N$ to be determined later. Then, under the assumptions in Theorem \ref{th1.3}, there holds that
	\begin{equation}\label{3.2}
		\begin{split}
			\Delta F-F_t \geq &\frac{(2-\delta)JF^2}{4n}+\frac{(2-\delta)J(J-2)F\lvert \nabla \rvert^2 }{2n}-\left( \frac{(2-\delta)J(N-a\ln D)}{n}+a\right)F\\
			&+\left\lbrace A+a-Ja+J-3-J\lvert \ln D\rvert+\frac{(2-\delta)J(J-2)(N-a\ln D)}{n}\right\rbrace\lvert \nabla f\rvert^2\\
			&-2\left\langle \nabla f, \nabla F\right\rangle + \frac{(2-\delta)J(2-J)^2\lvert \nabla f\rvert^4 }{4n}-\delta J\lvert \nabla f\rvert^4+F_1+F_2,
		\end{split}	
	\end{equation}
	where
	
	\begin{equation*}
		\begin{split}
			F_1=&
			f\left\lbrace (2-2J)\lvert \nabla f\rvert\lvert \nabla a\rvert+\frac{(2-\delta)J(J-2)(A-a)\lvert \nabla f\rvert^2}{2n}-\frac{(2-\delta)J(A-a)F}{2n}\right\rbrace\\
			&+f\left\lbrace \Delta a+a_t+\frac{(2-\delta)(A-a)(N-a\ln D)J}{n}\right\rbrace
		\end{split}
	\end{equation*}
	and
	\begin{equation*}
		\begin{split}
			F_2=&
			2\Delta b+2a(N+b)+2a_t\ln D-(A+a)(b+a\ln D)-\left( 1+J\lvert \ln D\rvert\right) \lvert\nabla a \rvert^2\\
			&-(2-J)\lvert\nabla b \rvert^2+(N-a\ln D)^2.
		\end{split}
	\end{equation*}
\end{lemma}
\begin{proof}[\textbf{Proof}]
	First of all, the equation (\ref{equ*})	is equivalent to
	\begin{equation}\label{3.3}
		\Delta f=f_t-\lvert\nabla f \rvert^2-af-a\ln D-b,
	\end{equation}
	then
	\begin{equation}\label{3.4}
		f_{tt}=\Delta f_t+af_t+a_tf+2\left\langle \nabla f, \nabla f_t\right\rangle+b_t+a_t\ln D.
	\end{equation}
	On the other hand, we have
	\begin{equation}\label{3.5}
		\begin{split}
			\Delta F=&\Delta\left( J\lvert \nabla f\rvert^2+(A+a)f+2(N+b)-2f_t\right)\\
			=&\lvert \nabla J\rvert^2\Delta J+J\Delta\lvert \nabla f\rvert^2+2\left\langle \nabla J,\nabla\lvert \nabla f\rvert^2\right\rangle+(A+a)\Delta f\\
			&+f\Delta a+2\left\langle\nabla a, \nabla f \right\rangle+2\Delta b-2\Delta f_t.
		\end{split}
	\end{equation}
	By Bochner formula, we know that there holds
	\begin{equation}\label{3.6}
		J\Delta\lvert \nabla f\rvert^2\geq2J\left\langle \nabla f, \nabla\Delta f\right\rangle+2J\lvert D^2f\rvert^2-2R_cJ\lvert \nabla f\rvert^2.
	\end{equation}
	We also have
	\begin{equation}\label{3.7}
		\begin{split}
			\left\langle \nabla f, \nabla F\right\rangle=&\left\langle\nabla f, \nabla J \right\rangle\lvert \nabla f\rvert^2+J\left\langle\nabla f, \nabla \lvert \nabla f\rvert^2\right\rangle+f\left\langle \nabla a, \nabla f\right\rangle+(A+a)\lvert \nabla f\rvert^2\\
			&+2\left\langle \nabla b, \nabla f\right\rangle-2\left\langle \nabla f_t, \nabla f\right\rangle
		\end{split}
	\end{equation}
	and
	\begin{equation}\label{3.8}
		aF=aJ\lvert \nabla f\rvert^2+(A+a)af+2(N+b)a-2af_t.
	\end{equation}
	For $F_t$, we have
	\begin{equation}\label{3.9}
		F_t=2J\left\langle \nabla f, \nabla f_t\right\rangle+J_t\lvert \nabla f\rvert^2+(A+a)f+fa_t+2b_t-2f_{tt}.
	\end{equation}
	Combining (\ref{3.4}), (\ref{3.6}), (\ref{3.7}), (\ref{3.8}) and (\ref{3.9}), then we have
	\begin{equation}\label{3.10}
		\begin{split}
			\Delta F-F_t\geq &\left[  2(A+a)-Ja)\right] \lvert \nabla f\rvert^2+\left[(2-2J)f+2-2J\ln D\right]\left\langle \nabla f, \nabla a\right\rangle \\
			& +(4-2J)\left\langle \nabla f, \nabla b\right\rangle +\lvert \nabla J\rvert^2\Delta J+2\left\langle\nabla f, \nabla J \right\rangle\lvert \nabla f\rvert^2+2J\lvert D^2f\rvert^2\\
			&-2R_cJ\lvert \nabla f\rvert^2+2\left\langle \nabla J,\nabla\lvert \nabla f\rvert^2\right\rangle-2\left\langle \nabla f, \nabla F\right\rangle-aF+f\Delta a\\
			&+a_tf+2a_t\ln D+2\Delta b+2(N+b)a-J_t\lvert \nabla f\rvert^2\\
			&-(A+a)\lvert \nabla f\rvert^2-(A+a)b-(A+a)a\ln D.
		\end{split}
	\end{equation}
	By Cauchy-Schwartz inequality, we can see that there hold
	\begin{equation}\label{3.11}
		(2-2J\ln D)\left\langle \nabla f, \nabla a\right\rangle\geq \left(-1-J\lvert \ln D\rvert\right)\left( \lvert \nabla f\rvert^2+\lvert \nabla a\rvert^2\right),
	\end{equation}
	\begin{equation}\label{3.12}
		(4-2J)\left\langle \nabla f, \nabla a\right\rangle\geq -\left(2-J\right)\left( \lvert \nabla f\rvert^2+\lvert \nabla b\rvert^2\right),
	\end{equation}
	\begin{equation}\label{3.13}
		2\left\langle \nabla J,\nabla\lvert \nabla f\rvert^2\right\rangle\geq-\delta J\lvert D^2f\rvert^2-\frac{4\lvert \nabla J\rvert^2\lvert \nabla f\rvert^2 }{\delta J},
	\end{equation}
	\begin{equation}\label{3.14}
		2\left\langle\nabla f, \nabla J \right\rangle\lvert\nabla f\rvert^2\geq-\delta J\lvert\nabla f\rvert^4-\frac{\lvert\nabla J\rvert^2\lvert \nabla f\rvert^2 }{\delta J},
	\end{equation}
	and
	\begin{equation}\label{3.15}
		\lvert D^2f\rvert^2\geq\frac{(\Delta f)^2}{n}.
	\end{equation}
	Combining the definition of $F$ and (\ref{3.3}), we have
	\begin{equation}\label{3.16}
		\Delta f=\frac{F}{2}+\frac{J-2}{2}\lvert \nabla f\rvert^2+\frac{(A-a)f}{2}+N-a\ln D.
	\end{equation}
	Substituting (\ref{3.11}), (\ref{3.12}), (\ref{3.13}), (\ref{3.14}) into (\ref{3.10}), we have
	\begin{equation*}
		\begin{split}
			\Delta F-F_t\geq &\left( A+a+J-Ja-3-J\lvert \ln D\rvert \right)\lvert \nabla f\rvert^2\\
			&+(2-2J)f\lvert\nabla f\rvert\lvert\nabla a\rvert+(2-\delta)J\lvert D^2f\rvert^2-aF\\
			&-2\left\langle \nabla f, \nabla F\right\rangle+f\Delta a+2\Delta b+a_tf+2a_t\ln D \\
			&+2a(N+b)-(b+a\ln D)(A+a)-\left( 1+J\lvert \nabla a\rvert^2\right)\\
			&-(2-J)\lvert \nabla b\rvert^2+\left(\Delta J-2R_cJ-5\delta^{-1}\frac{\lvert \nabla J\rvert^2 }{J}-J_t\right)\lvert \nabla f\rvert^2.
		\end{split}
	\end{equation*}
	Then by Lemma \ref{l3.1}, for $r=1$, we can choose $0<\underline{J}=\underline{J_1}\leq J\leq1$ such that $J$ solves the equation in Lemma \ref{l3.1}. In this situation, we have
	\begin{equation}\label{3.17}
		\begin{split}
			\Delta F-F_t\geq &\left( A+a+J-Ja-3-J\lvert\ln D\rvert \right)\lvert\nabla f\rvert^2\\
			&+(2-2J)f\lvert \nabla f\rvert\lvert\nabla a\rvert+(2-\delta)J\lvert D^2f\rvert^2-aF\\
			&-2\left\langle \nabla f, \nabla F\right\rangle+f\Delta a+2\Delta b+a_tf+2a_t\ln D\\
			&+2a(N+b)-(b+a\ln D)(A+a)\\
			&-\left( 1+J\lvert \nabla a\rvert^2\right)-(2-J)\lvert \nabla b\rvert^2.
		\end{split}
	\end{equation}
	Substituting (\ref{3.15}) and (\ref{3.16}) into (\ref{3.17}) and noting $\left( \frac{(A-a)f}{2}\right) ^2\geq0$, we obtain (\ref{3.2}).
\end{proof}

Now we can give:
\begin{proof}[\textbf{Proof of Theorem \ref{th1.3}}]
	Since on $B(x,1)\times(0,\infty)$, $a,b,\lvert\nabla a\rvert,\lvert\nabla b\rvert$ and $\lvert a_t \rvert$ are bounded,  $\Delta b$ is bounded from below and $0<J\leq1$, then there exists a constant $N\leq0$ such that $F_2\geq0$ and $N<a\ln D$ on $B(x,1)\times(0,\infty)$. Moreover, we choose a constant $A$ such that $A>(a)^+$ on $B(x,1)\times(0,\infty)$.
	For this situation, from (\ref{3.2}) we have
	\begin{equation*}
		\begin{split}
			\Delta F-F_t\geq &\frac{(2-\delta)JF^2}{4n}+\frac{(2-\delta)J(J-2)F\lvert \nabla f\rvert^2 }{2n}-\left( \frac{(2-\delta)J(N-a\ln D)}{n}+a\right)F\\
			&+\left\lbrace A+a-Ja+J-3-J\lvert \ln D\rvert+\frac{(2-\delta)J(J-2)(N-a\ln D)}{n}\right\rbrace \lvert \nabla f\rvert^2\\
			&-2\left\langle \nabla f, \nabla F\right\rangle +\frac{(2-\delta)J(2-J)^2\lvert \nabla f\rvert^4 }{4n}-\delta J\lvert\nabla f\rvert^4+F_1.
		\end{split}	
	\end{equation*}
	For $N$, furthermore we can require that
	$$ A+a-Ja+J-3-J\lvert \ln D\rvert+\frac{(2-\delta)J(J-2)(N-a\ln D)}{n}\geq0$$
	since we have $J-2<0<J$, $N-a\ln D<0$ and $A+a-3\geq0$. In this situation, there holds
	\begin{equation*}
		\begin{split}
			\Delta F-F_t\geq &\frac{(2-\delta)JF^2}{4n}+\frac{(2-\delta)J(J-2)F\lvert \nabla f\rvert^2 }{2n}-\left( \frac{(2-\delta)J(N-a\ln D)}{n}+a\right)F\\
			&-2\left\langle \nabla f, \nabla F\right\rangle +\left( \frac{(2-\delta)J(2-J)^2 }{4n}-\delta J\right) \lvert \nabla f\rvert^4+F_1.
		\end{split}	
	\end{equation*}
	Moreover, since $0<\delta\leq\frac{2}{1+4n}$, then
	$$\frac{(2-\delta)J(2-J)^2 }{4n}-\delta J\geq0.$$
	So, finally we have
	\begin{equation}\label{3.19}
		\begin{split}
			\Delta F-F_t\geq&\frac{(2-\delta)JF^2}{4n}+\frac{(2-\delta)J(J-2)F\lvert \nabla f\rvert^2 }{2n}\\
			&-\left( \frac{(2-\delta)J(N-a\ln D)}{n}+a\right)F -2\left\langle \nabla f, \nabla F\right\rangle+F_1.
		\end{split}	
	\end{equation}
	Next we can choose a cut-off function $\phi$ on $B_1$ as in Lemma \ref{l1.1}. Then for any $T>0$, we assume $(x_0, t_0)$ is the maximum point of $t\phi^2F$ on $B_1\times(0,T]$. It follows from maximum principle that there holds that at $(x_0, t_0)$
	\begin{equation}\label{3.20}
		\nabla\left( \phi^2F\right)=0,\quad\quad \Delta\left(t\phi^2F \right)\leq0,\quad\quad\left( t\phi^2F\right)_t\geq0,
	\end{equation}
	then we also obtain $\nabla F=-\frac{2F\nabla\phi}{\phi}$ from the first identity.
	
	From now on, all the calculations are considered at the point $(x_0, t_0)$ and for simplicity, we omit it. By maximum principle and direct calculations, we have
	\begin{equation}\label{3.21}
		\begin{split}
			&t^2\phi^4\left( \Delta-\frac{\partial}{\partial t}\right)F+2t^2\phi^3F\Delta\phi+2t^2\phi^2F\lvert\nabla\phi \rvert^2-8t^2\phi^2F\lvert\nabla\phi \rvert^2-t\phi^4F\\
			=&t\phi^2\left( \Delta-\frac{\partial}{\partial t}\right)\left(t\phi^2F \right)\leq0.
		\end{split}
	\end{equation}
	Substituting (\ref{3.19}) into (\ref{3.21}), we deduce
	\begin{equation}\label{3.22}
		\begin{split}
			&\frac{(2-\delta)J(t\phi^2F)^2}{4n}+\frac{(2-\delta)J(J-2)t^2\phi^4F\lvert \nabla f\rvert^2 }{2n}-\left( \frac{(2-\delta)J(N-a\ln D)}{n}+a\right)t^2\phi^4F\\
			&+4t^2\phi^3F\left\langle \nabla f, \nabla\phi\right\rangle+2t^2\phi^3F\Delta\phi -6t^2\phi^2F\lvert\nabla\phi \rvert^2-t\phi^4F+t^2\phi^4F_1\leq0.
		\end{split}
	\end{equation}
	\noindent\textbf{Case I.} \, $t^2\phi^4F_1\leq0$, i.e. $t\phi^2F_1\leq0$.
	
	Since $f\leq0$, then we have
	\begin{equation*}
		\begin{split}
			&t\phi^2\left[ (2-2J)\lvert \nabla f\rvert\lvert \nabla a\rvert+\frac{(2-\delta)J(J-2)(A-a)\lvert \nabla f\rvert^2}{2n}-\frac{(2-\delta)J(A-a)F}{2n}\right] \\
			&+t\phi^2\left[ \Delta a+a_t+\frac{(2-\delta)(A-a)(N-a\ln D)J}{n}\right] \geq0.
		\end{split}	
	\end{equation*}
	Noting $N-a\ln D<0$, then we have
	\begin{equation*}
		(2-2J)\lvert \nabla f\rvert\lvert \nabla a\rvert+\frac{(2-\delta)J(J-2)(A-a)\lvert\nabla f\rvert^2}{2n}-\frac{(2-\delta)J(A-a)F}{2n}
		+\Delta a+a_t\geq0.
	\end{equation*}
	By using Cauchy-Schwartz inequality we can infer that there holds
	\begin{equation*}
		(2-2J)\lvert\nabla f\rvert\lvert \nabla a\rvert\leq\frac{(2-\delta)J(2-J)(A-a)\lvert\nabla f\rvert^2}{2n}+\frac{2n(1-J)^2\lvert \nabla a\rvert^2}{(2-\delta)J(2-J)(A-a)},
	\end{equation*}
	and this leads to that
	\begin{equation*}
		\frac{(2-\delta)J(A-a)F}{2n}\leq\frac{2n(1-J)^2\lvert \nabla a\rvert^2}{(2-\delta)J(A-a)}+\Delta a+a_t.
	\end{equation*}
	Since $0<\underline{J}\leq J\leq1$ and $$\underline{J}\lvert \nabla f\rvert^2+(A+a)f+2(N+b)-2f_t\leq F,$$
	at the maximum point $(x_0, t_0)$ we have
	\begin{equation*}
		t\phi^2F\leq\frac{2nt\phi^2}{(2-\delta)\underline{J}\left( A-(a)^+\right) }\left\lbrace\frac{2n\left( 1-\underline{J}\right) ^2\left( \lvert \nabla a\rvert^2\right) ^+}{(2-\delta)\underline{J}\left(A-(a)^+\right)}+\left( \Delta a+a_t\right)^+  \right\rbrace.
	\end{equation*}
	Then we obtain on $B\left( x,\frac{1}{2}\right) \times(0,\infty)$
	\begin{equation}\label{3.18}
		\begin{split}
			&\underline{J}\lvert \nabla f\rvert^2+(A+a)f+2(N+b)-2f_t\\
			&\leq\frac{2n}{(2-\delta)\underline{J}\left( A-(a)^+\right) }\left\lbrace\frac{2n\left( 1-\underline{J}\right) ^2\left( \lvert \nabla a\rvert^2\right) ^+}{(2-\delta)\underline{J}\left(A-(a)^+\right)}+\left( \Delta a+a_t\right)^+ \right\rbrace.
		\end{split}
	\end{equation}
	
	\noindent\textbf{Case II.} \, $t^2\phi^4F_1\geq0$.

	Since $\lvert \Delta\phi\rvert \leq C$, $\lvert\nabla\phi \rvert^2\leq C$, $(N-a\ln D)^+<0$ and $\phi^2\geq\phi^3$, then, by (\ref{3.22}) we can obtain
	\begin{equation*}
		\begin{split}
			&\frac{(2-\delta)J(t\phi^2F)^2}{4n}+\frac{(2-\delta)J(J-2)t^2\phi^4F\lvert \nabla f\rvert^2 }{2n}-(a)^+t^2\phi^4F\\
			&-4t^2\phi^3FC^{\frac{1}{2}}\lvert \nabla f\rvert -2Ct^2\phi^2F-6Ct^2\phi^2F-t\phi^4F\leq0,
		\end{split}
	\end{equation*}
	i.e.
	\begin{equation}\label{3.23}
		\begin{split}
			&\frac{(2-\delta)J(t\phi^2F)^2}{4n}+\frac{(2-\delta)J(J-2)t^2\phi^4F\lvert\nabla f\rvert^2 }{2n}-(a)^+t^2\phi^4F\\
			&-4t^2\phi^3FC^{\frac{1}{2}}\lvert \nabla f\rvert -8Ct^2\phi^2F-t\phi^4F\leq0.
		\end{split}
	\end{equation}
	By Cauchy-Schwartz inequality, we have
	\begin{equation}\label{3.24}
		\frac{(2-\delta)J(2-J)\phi^4\lvert \nabla f\rvert^2 t^2F}{2n}+\frac{8nC\phi^2t^2F}{(2-\delta)J(2-J)}\geq4C^\frac{1}{2}\phi^3\lvert\nabla f\rvert t^2F.
	\end{equation}
	Substituting (\ref{3.24}) into (\ref{3.23}) gives
	\begin{equation}\label{3.25}
		\frac{(2-\delta)J(t\phi^2F)^2}{4n}-(a)^+t^2\phi^4F-\frac{8nCt^2\phi^2F}{(2-\delta)J(2-J)}-8Ct^2\phi^2F-t\phi^4F\leq0.
	\end{equation}
	Without loss of generality, we may assume $F>0$, since otherwise the result is trivial. Then dividing by $t\phi^2F$ on both sides of (\ref{3.25}) yields
	\begin{equation}\label{3.26}
		\frac{(2-\delta)Jt\phi^2F}{4n}-(a)^+t\phi^2-\frac{8nCt}{(2-\delta)J(2-J)}-8Ct-\phi^2\leq0,
	\end{equation}
	i.e.
	\begin{equation}\label{3.27}
		\begin{split}
			t\phi^2F&\leq\frac{4n}{(2-\delta)J}\left\lbrace (a)^+t\phi^2+\frac{8nCt}{(2-\delta)J(2-J)}+8Ct+\phi^2\right\rbrace\\
			&\leq\frac{4n}{(2-\delta)\underline{J}}\left\lbrace (a)^+T\phi^2+\frac{8nCT}{(2-\delta)\underline{J}}+8CT+\phi^2\right\rbrace.
		\end{split}
	\end{equation}
	Since for any $0<t\leq T$, we have
	$$tF\lvert_{B_{\frac{1}{2}}\times(0, T]}=t\phi^2F\lvert _{B_{\frac{1}{2}}\times(0, T]}\leq\sup\limits_{B_1\times(0, T]}t\phi^2F=t\phi^2F\lvert _{B_1\times(0, T]}(x_0, t_0),$$
	i.e., on $B_{\frac{1}{2}}\times(0, T]$, by (\ref{3.27}) there holds true that
	\begin{equation}\label{3.28}
		TF\leq\frac{4n}{(2-\delta)\underline{J}}\left\lbrace (a)^+T+\frac{8nCT}{(2-\delta)\underline{J}}+8CT+1\right\rbrace.
	\end{equation}
	Finally we have
	\begin{equation}\label{3.29}
		\underline{J}\lvert\nabla f\rvert^2+(A+a)f+2(N+b)-2f_t\leq F\leq\frac{4n}{(2-\delta)\underline{J}}\left\lbrace (a)^++\frac{8nC}{(2-\delta)\underline{J}}+8C+\frac{1}{T}\right\rbrace.
	\end{equation}
	Now combining (\ref{3.18}) and (\ref{3.29}) gives (\ref{1.7}), but we need to choose a appropriate $k(n, p)$ such that Lemma \ref{l1.1} and Lemma \ref{3.1} are valid. Hence we accomplish the proof of Theorem \ref{th1.3} since $T$ is arbitrary.
\end{proof}

\begin{remark}\label{r3.1}
If $a(x, t)$ is a constant, e.g., $a\equiv A$, then we can see that $F_2\equiv0$, hence in this situation it's unnecessary to assume $u$ has a upper bound by letting $D\equiv1$.
\end{remark}

\noindent {\it\textbf{Acknowledgements}}: The authors are supported partially by NSFC grant (No.11731001), partially by NSFC grant (No.11971400) and Guangdong Basic and Applied Basic Research Foundation Grant (No. 2020A1515011019). The authors are grateful to Professor Bing Wang for many constructive suggestions.

\bibliographystyle{amsalpha}
	
\end{document}